\documentclass[10pt]{amsart}

\usepackage{amsthm,amsmath,amssymb}
\usepackage{mathtools}
\usepackage{graphicx}
\usepackage[comma, sort&compress]{natbib}
\usepackage{tabularx}
\usepackage{multirow}
\usepackage{amssymb}
\usepackage[utf8]{inputenc}
\usepackage{setspace}
\usepackage{subfig}
\usepackage{leftidx}

\usepackage{paralist}
\RequirePackage{hyperref}

\newtheorem{theorem}{Theorem}

\newtheorem{lemma}[theorem]{Lemma}

\newtheorem{corollary}[theorem]{Corollary}

\theoremstyle{definition}

\theoremstyle{remark}

\numberwithin{equation}{section}
\numberwithin{theorem}{section}
\numberwithin{example}{section}
\numberwithin{definition}{section}
\numberwithin{figure}{section}

\DeclareMathOperator{\diag}{diag}

\DeclareMathOperator{\iden}{{\bf I}_m}
\DeclareMathOperator{\cormat}{{\bf R}}

\newcommand{\secref}[1]{Section~\ref{sec:#1}}
\newcommand{\secsref}[1]{Sections~\ref{sec:#1}}
\newcommand{\secssref}[1]{\ref{sec:#1}}

\newcommand{\appref}[1]{Appendix~\ref{app:#1}}

\newcommand{\lemref}[1]{Lemma~\ref{lem:#1}}
\newcommand{\lemsref}[1]{Lemmas~\ref{lem:#1}}
\newcommand{\lemssref}[1]{\ref{lem:#1}}

\newcommand{\thmref}[1]{Theorem~\ref{thm:#1}}

\title[]{Asymptotic power of Rao's score test for independence in high dimensions}

\author[D.~Leung]{Dennis Leung}
\address{Department of Statistics, Chinese University of Hong Kong, Shatin, Hong Kong}
\email{dmhleung@uw.edu}

\author[Q.M.~Shao]{Qi-Man Shao}
\address{Department of Statistics, Chinese University of Hong Kong, Shatin, Hong Kong}
\email{qmshao@sta.cuhk.edu.hk}

\begin{document}

\begin{abstract}
Let $\cormat$ be the Pearson correlation matrix of $m$ normal random variables. The Rao's score test for the independence hypothesis $H_0 : \cormat = \iden$, where $\iden$ is the identity matrix of dimension $m$, was first considered by  \citet{Schott2005} in the high dimensional setting.  In this paper, we study the asymptotic exact power function of this test, under an asymptotic regime in which both $m$ and the sample size $n$ tend to infinity with the ratio $m/n$ upper bounded by a constant. In particular, our result implies that the Rao's score test is minimax rate-optimal for detecting the dependency signal $\|\cormat - \iden\|_F$ of order $\sqrt{m/n}$, where $\|\cdot\|_F$ is the matrix Frobenius norm.

\end{abstract}



\maketitle

\section{Introduction} \label{sec:intro}
Let $(X_1, \dots, X_m)'$ be an $m$-variate normal vector with population Pearson correlation matrix denoted by ${\bf R} = (\rho_{pq})_{1 \leq p, q \leq m}$. Suppose we observe $n$ independent samples $X_{p1}, \dots, X_{pn}$ for each component $X_p$, $1 \leq p \leq m$. When the dimension $m$ can be larger than the sample size $n$, \citet{Schott2005} was the first to consider the Rao's score statistic
\begin{equation} \label{Schottstat}
T = \sum_{ 1 \leq p < q \leq m} \hat{\rho}_{pq}^2,
\end{equation}
for testing the independence null hypothesis
\begin{equation}\label{nullIndep}
H_0: {\bf R} = {\bf I}_m,
\end{equation}
where $\hat{\rho}_{pq}$, $1 \leq p \not= q \leq m$ is the sample correlation of the pair $(X_p, X_q)$ computed from the data, and ${\bf I}_m$ is the $m$-by-$m$ identity matrix. It was shown to be asymptotically normal under $H_0$ as both $m$ and $n$ go to infinity with the ratio $m/n$ converging to a positive constant.
%This result has subsequently been refined by \citet{ChenShao}, who proved a Berry-Esseen bound for the convergence of $T$ to the standard normal after suitable renormalization as $m, n \longrightarrow \infty$, with no restrictions on the ratio $m/n$.
The purpose of this paper is to complement the theoretical study of $T$ by investigating its power under alternatives of the form
\[
H_1: \cormat \in \Theta(b),
\]
where for any constant $b > 0$ and matrix Frobenius norm $\|\cdot\|_F$, we define the set of Pearson correlation matrices
\begin{equation} \label{alt}
\Theta(b) := \{{\bf R}: \|{\bf R} - {\bf I}_m \|_F \geq b \sqrt{m/n}, \ \ \diag(\cormat) = \iden \},
\end{equation}
which comprises a composite alternative hypothesis delineated by a signal size $\|{\bf R} - {\bf I}_m\|_F$ of  order no less than $\sqrt{m/n}$.

There are three major approaches to   testing independence with growing dimension $m$ in the literature, to the best of our knowledge. The first is the statistic $T$  considered in this paper. Being a ``sum" of squared pairwise sample correlation as in \eqref{Schottstat}, it is good at detecting diffuse dependency among many pairs of variables. Such dependency is most naturally described by the signal $\|{\bf R} - {\bf I}_m\|_F$. In fact, the main result in this paper will show that $T$ is minimax rate optimal for detecting such signal. The second approach considers the ``max" statistic,
\[
\max_{1 \leq p < q \leq m} \hat{\rho}_{pq}^2.
\]
Following many previous works \citep{jiang2004, Liu2008, li2010, zhou2007, li2010followup}, \citet{cai2011} showed that it
admits an asymptotic Gumbel distribution under $H_0$ in the ultra high dimensional regime when $m$ can be as large as $e^{n^c}$ for some constant $0 < c < 1$, as $m, n \longrightarrow \infty$. Naturally, it is good at detecting a structured alternative whose population correlation matrix $\cormat$ has sparse non-zero off-diagonal entries with considerable magnitudes. Both the ``sum" and ``max" approaches base their test on forming intuitive statistics that measure the overall dependency among the $m$ variables, with their respective non-parametric extensions; see \citet{leung2015testing} and \citet{han2014distribution}. The third is likelihood ratio test (LRT), which is well-known to give implementable test only if the dimension $m$ is smaller than $n$. Despite this limitation, \citet{JiangQi} showed the LRT statistic to be asymptotically normal when $m, n \longrightarrow \infty$, as long as $m + 4$ is less than $n$.

We remark that the derivation of \eqref{Schottstat} as the Rao's score statistic involves taking derivatives of the log-normal likelihood with respect to the mean vector and the precision matrix. The interested reader is referred to Appendix A in \citet{leung2015testing} for those calculations.

\section{Notations and main results}\label{sec:main}

For any positive integer $k$, $[k]$ is defined as the set $\{1, \dots, k\}$. $\mathcal{S}_k$ is the symmetric group of order $k$. Depending on the context, its elements will sometimes be treated as permutation functions on $k$ elements, or simply permutations of the set $[k]$. $C$ always denotes a positive constant that is universal, i.e, its value may change from place to place but does not depend on $m$ and $n$. ``$a \lesssim b$" means that $a \leq Cb$ for some constant $C >0$. $\mathbb{E}[\cdot]$, $\text{Var}[\cdot]$ and $P[\cdot]$ are expectation, variance and probability operators respectively.

 In this paper we shall \emph{always} assume that, for all $1 \leq p \leq m$,
$\text{Var}[X_p] = 1$ and  $\mathbb{E}[X_p] = 0$. Thus, for a duple $(p, q) \in [m] \times [m]$,  $\mathbb{E}[X_p X_q] =  \rho_{pq}$, and  its corresponding squared sample correlation  is defined as
\begin{equation} \label{r2}
\hat{\rho}_{pq}^2 :=  \frac{S_{pq}^2}{S_{pp}S_{qq}} = f\left( S_{pp}, S_{qq}, S_{pq}\right),
\end{equation}
where $f : \mathbb{R}_{> 0}^2 \times \mathbb{R} \longrightarrow \mathbb{R}$ is the function  \begin{equation} \label{fun}
f(u_1, u_2, u_3) := u_1^{-1} u_2^{-1} u_3^2,
\end{equation}
and
\begin{equation} \label{nice}
%{\bf S}= (S_{pq})_{1 \leq p, q \leq m} \text{ with }
S_{pq} :=  \frac{\sum_{i = 1}^nX_{pi}X_{qi}}{n}.
\end{equation}
We will also use
\[
\bar{S}_{pq} := S_{pq} - \rho_{pq}
\]
 to denote the centered sample covariance. Imposing the assumption $\text{Var}[X_p] = 1$ is always permitted, even if we use the more general form of Pearson correlations with all sample covariances $S_{pq}$  defined alternatively as
 \begin{equation} \label{nice2}
 \frac{\sum_{i = 1}^n (X_{pi} - n^{-1}\sum_{j = 1}^nX_{pj}  )(X_{qi}-  n^{-1}\sum_{j = 1}^nX_{qj})}{n -1}
 \end{equation}
  in \eqref{r2}, since the distribution of $\hat{\rho}_{pq}$ is invariant to the scaling of variables. Under normality, the restrictions $\mathbb{E}[X_p] = 0$ and \eqref{nice}   can be still be assumed without forgoing any generality of our results to follow; see the classical result in \citet[Theorem 3.3.2]{TWanderson}.

According to \citet[Theorem 2.2]{ChenShao} who refined the asymptotic result of \cite{Schott2005} under $H_0$, for a given $\alpha \in (0, 1)$, a test of asymptotic level $\alpha$ based on \eqref{Schottstat} is given as
\begin{equation} \label{ourtestform}
\psi  = I\left( T - \frac{m(m-1)}{2 n}  > \frac{m}{n}z_{\alpha} \right),
\end{equation}
where $I(\cdot)$ is the indicator function , $z_{\alpha} := \bar{\Phi}^{-1}( \alpha)$ ,  and $\Phi$ and  $\bar{\Phi}(x) := 1 - \Phi(x)$ are respectively the cumulative distribution function and tail probability of a  standard normal variate. Below, $\mathbb{E}_{\cormat}[\cdot]$ simply emphasizes that the expectation is taken with respect to a particular correlation matrix $\cormat \in \Theta(b)$.
%
%To state our main result on its minimax power, for any constant $b > 0$, we define the set of Pearson correlation matrices
%\begin{equation} \label{alt}
%\Theta(b) = \{{\bf R}: \|{\bf R} - {\bf I}_m \|_F \geq b \sqrt{m/n} \}
%\end{equation}
%which comprises a composite alternative hypothesis delineated by a signal $\|{\bf R} - {\bf I}_m\|_F$ of size larger than the order $O(\sqrt{m/n})$.

\begin{theorem}[Main result: asymptotic power]\label{thm:main} Suppose $m, n \longrightarrow \infty$ such that $\frac{m}{n} \leq \kappa$ for some constant $\kappa < \infty$.  For any significance level $\alpha \in (0, 1)$, the asymptotic  power of $\psi$ is given as
\[
\lim_{n \rightarrow \infty} \inf_{\Theta(b)} \mathbb{E}_{\cormat}[\psi] = \bar{\Phi}(z_{\alpha} - 2^{-1} b^2).
\]

\end{theorem}

This theorem resembles \citet[Theorem $4$]{CaiMa}, in which the different problem of testing $H_0: { \boldsymbol\Sigma} = \iden$, where ${ \boldsymbol\Sigma}$ is the \emph{covariance} matrix of $(X_1, \dots, X_m)'$, is studied. Despite this, Theorem $1$ and Remark $1$ in their paper indicate that a matching lower bound on the detectable signal size as measured by $\|\cormat - \iden\|_F$ can be established for our problem \eqref{nullIndep}, which we restate next  for our readers' convenience. We add that \thmref{main} is slightly weaker than the parallel result of \citet{CaiMa} in that an upper bound on the ratio $m/n$ is imposed, which we believe to be merely a proof artifact not necessary for the theorem to hold. Discussion on this will be deferred later.

\begin{theorem} [Matching lower bound, \citet{CaiMa}] \label{thm:matchingbdd}
Let $0 < \alpha < \beta < 1$. Suppose $m, n \longrightarrow \infty$ such that $\frac{m}{n} \leq \kappa$ for some constant $\kappa < \infty$. Then there exists a constant $b = b(\kappa, \beta - \alpha) < 1$, such that
\[
\limsup_{n \longrightarrow \infty} \inf_{\Theta(b)} \mathbb{E}_{\cormat}[\phi] <\beta
\]
for any test $\phi$ with significance level $\alpha$ for testing $H_0$.
\end{theorem}

The lower bound result says that no $\alpha$-level test for $H_0$ can achieve a preset target power if the signal size $\|\cormat - \iden\|_F$ falls below a certain threshold modulo the separation rate $\sqrt{m/n}$ . Our main result in \thmref{main} hence suggests that our test $\psi$ is ``rate" optimal when the ratio $m/n$ is bounded, since the asymptotic  power $\lim_{n \longrightarrow \infty} \inf_{\Theta(b)} \mathbb{E}_{\cormat}[\psi]$ tends to one as $b \longrightarrow \infty$.

Although the result in \thmref{main} is neat, its proof, which occupies the rest of this paper, is quite involved. As it will become clear later, this is because our  statistic $T$ is constructed with Pearson correlations whose higher order moment properties involve a lot of computations to be understood; see \citet[Section 7]{hotelling} for classical work on this. At some point in this paper we will use \texttt{mathematica} to help us with certain symbolic calculations.
We shall begin with a Taylor expansion of the expression  for $\hat{\rho}^2_{pq}$ in terms of the function $f$ in \eqref{r2}. We need the multi-index notations: For a vector ${\boldsymbol \lambda} = (\lambda_1, \dots, \lambda_k)$ of $k$ non-negative integers, ${\boldsymbol \lambda}! = \lambda_1! \dots \lambda_k!$ and  $|{\boldsymbol \lambda}| = \lambda_1 + \dots + \lambda_k$, and if $g= g(u_1, \dots, u_k)$ is a function in $k$ arguments, $\partial^{\boldsymbol \lambda}g (\tilde{u}_1, \dots, \tilde{u}_k)= \frac{\partial^{|{\boldsymbol \lambda}|} g}{\partial u_1^{\lambda_1}\dots\partial u_k^{\lambda_k}}\big|_{u_i = \tilde{u}_i}$ is its partial derivative with respect to  $\boldsymbol \lambda$ evaluated at the point $ (\tilde{u}_1, \dots \tilde{u}_k) $. Since $\rho_{pq}^2 = f(1, 1, \rho_{pq}) = f(\rho_{pp}, \rho_{qq}, \rho_{pq})$, by Taylor's theorem, for each pair $1 \leq p \not = q \leq m$,
\begin{equation}\label{Taylor}
\hat{\rho}_{pq}^2 - \rho_{pq}^2
 =  \sum_{\substack{{\boldsymbol \lambda} \in \mathbb{N}^3_{\geq 0}:  \\1 \leq |{\boldsymbol \lambda}| \leq 4  }} \frac{\partial^{\boldsymbol \lambda} f(1, 1, \rho_{pq})}{{\boldsymbol \lambda}!}\bar{S}_{pp}^{\lambda_1} \bar{S}_{qq}^{\lambda_2}\bar{S}_{pq}^{\lambda_3}+   III_{pq} \ \ \text{  a.s.},
\end{equation}
where
\begin{equation} \label{remainder}
III_{pq} := \sum_{\substack{{\boldsymbol \lambda} \in \mathbb{N}^3_{\geq 0}:  \\ |{\boldsymbol \lambda}| = 5  }} \frac{{(\rho_{pq} + k_{pq}\bar{S}_{pq})}^{2 - \lambda_1 }\bar{S}_{pp}^{\lambda_1} \bar{S}_{qq}^{\lambda_2}\bar{S}_{pq}^{\lambda_3}}{{(1 + k_{pq}\bar{S}_{pp})}^{1 + \lambda_2 } {( 1 +k_{pq}\bar{S}_{qq} )}^{1 + \lambda_3 }},
\end{equation}
for  some $k_{pq} = k_{pq}(S_{pp}, S_{qq}, S_{pq}) \in (0, 1)$,  is the remainder in Lagrange's form. The ``almost surely" qualifier  is in  \eqref{Taylor}  because  on an event of measure zero, either $S_{pp}$ or $S_{qq}$ may be zero, in which case the Taylor's theorem doesn't apply since $f$ is defined on $\mathbb{R}_{> 0}^2 \times \mathbb{R}$. Our proof depends crucially on recognizing that, when ${\boldsymbol \lambda} = (\lambda_1, \lambda_2, \lambda_3) = (0, 0, 2)$,
\begin{align*}
& \frac{\partial^{\boldsymbol \lambda }f(1, 1, \rho_{pq})}{{\boldsymbol \lambda}!}
\bar{S}_{pp}^{\lambda_1} \bar{S}_{qq}^{\lambda_2}\bar{S}_{pq}^{\lambda_3} = \bar{S}_{pq}^2 \\
%&= \left(\frac{\sum_{i = 1}^n (X_{pi}X_{qi} - \rho_{pq})}{n}\right)^2 \\
&= \frac{\sum_{i = 1}^n (X_{pi}X_{qi} - \rho_{pq})^2}{n^2} + \frac{2 \sum_{ 1 \leq i < j \leq n} (X_{pi}X_{qi} - \rho_{pq})(X_{pj}X_{qj} - \rho_{pq})}{n^2} ,
\end{align*}
in light of \lemref{df} which specifies the partial derivatives of $f$.
One can then equivalently write \eqref{Taylor} as
\begin{equation}  \label{2ndTaylor}
\hat{\rho}_{pq}^2- \rho_{pq}^2 = I_{pq} + II_{ pq} + III_{pq},
\end{equation}
where
\begin{equation}
I_{pq} := \frac{2  \sum_{ 1 \leq i < j \leq n} (X_{pi}X_{qi} - \rho_{pq})(X_{pj}X_{qj} - \rho_{pq})}{n^2}, \text{ and} \label{I_pq}
\end{equation}
\begin{multline}
II_{pq} := \frac{\sum_{i = 1}^n (X_{pi}X_{qi} - \rho_{pq})^2}{n^2} + \sum_{\substack{{\boldsymbol \lambda} \in \mathbb{N}^3_{\geq 0}:  \\1 \leq |{\boldsymbol \lambda}| \leq 4  \\ {\boldsymbol \lambda} \not = (0, 0 , 2)}} \frac{\partial^{\boldsymbol \lambda} f(1, 1, \rho_{pq})}{{\boldsymbol \lambda}!}\bar{S}_{pp}^{\lambda_1} \bar{S}_{qq}^{\lambda_2}\bar{S}_{pq}^{\lambda_3}. \label{II_pq}
\end{multline}
Defining $I  := \sum_{ 1\leq p < q \leq m} I_{pq}$, $II := \sum_{ 1\leq p < q \leq m} II_{pq}$ and $III :=  \sum_{ 1\leq p < q \leq m} III_{pq}$ by summing over all $1\leq p < q \leq m$, from \eqref{2ndTaylor} one can write
\begin{equation} \label{TandOthers}
T  -  \frac{m(m-1)}{2n} - 2^{-1}\|{\bf R} - \iden\|_F^2= I + \left( II  -  \frac{m(m-1)}{2n}\right) + III,
\end{equation}
realizing that $2^{-1}\|{\bf R} -{\bf I}_m\|_F^2 = \sum_{1 \leq p< q \leq m} \rho_{pq}^2$. We are now in the position to  introduce three supporting lemmas that are the building blocks of \thmref{main}. The first lemma gives a Berry-Esseen bound for the cumulative distribution function of the term $I$ with $\Phi(\cdot)$ after standardization. This will ultimately drive the form of our power function in \thmref{main}. The next two lemmas control the variability of the extra terms,  $( II  -  \frac{m(m-1)}{2n})$ and $III$. From now on for the rest of this paper all the big $O$, little $o$ notations are with respect to our considered asymptotic regime $m,n \longrightarrow \infty$, $m/n \leq \kappa$.

\begin{lemma} [Berry Esseen theorem for $I$] \label{lem:BerryEsseen}
The following are true for $I$:
\begin{enumerate}
\item Variance:
\begin{equation*}
 \text{Var}[I]  = \mathbb{E}[I^2] =
% \frac{m(m-1)n(n-1)}{n^4} + \frac{ O(   m  \|\cormat - \iden\|_F^2 + \|\cormat - \iden\|_F^4)}{n^2} \\
 \frac{m^2}{n^2} + o\left(\frac{m^{2(1 - \gamma)}}{n^2}\right) \sum_{k = 0}^2 \|\cormat - \iden\|_F^{2k}
\end{equation*}
for any $0 < \gamma < 1/2$.
\item Berry-Esseen bound:
\begin{multline*} \label{BEbddI}
\sup_{t \in \mathbb{R}} \left|P\left(   \frac{I}{\sqrt{\text{Var(I)}}}
 \leq t\right) - \Phi(t)\right| \lesssim \left\{ \frac{o(m^4/n^4) \sum_{k = 0}^8 \|\cormat - \iden\|_F^k}{ \text{Var}(I)^2}\right\}^{1/5}. \\
% \leq C \left\{ \frac{m^4 n^{-5}(  \sum_{k = 0}^4 \|\cormat - \iden\|^k_F) + m^3n^{-4}+ n^{-4}(\sum_{k = 0}^3 m^k\|\cormat - \iden\|_F^{2(4-k)} + m^3 \|\cormat - \iden\|_F)}{\text{Var}(I)^{2}}\right\}^{1/5}\\
\end{multline*}
\end{enumerate}
\end{lemma}

\begin{lemma}[Bound on the 2nd moment of $II  -  \frac{m(m-1)}{2n}$ ]\label{lem:varbdd}
\begin{multline}
\mathbb{E}\left[\left(II - \frac{m(m-1)}{2n}\right)^2\right] \lesssim\\
% \frac{\|R - I\|^4_F }{n} + \frac{m^2}{n^3}\left(1 + \|R - I\|_F^2\right)  +
% \frac{m^4}{n^5} +  \left(\frac{m^3}{n^4} + \frac{1}{n^2}\right)\|R - I\|_F +
% \frac{\|R  - I\|_F^2 }{n} \\
 \frac{\|\cormat- \iden\|_F^2 + \|\cormat - \iden\|_F^4}{n} + o\left(\frac{m^{2(1- \gamma)}}{n^2}\right) \sum_{k = 0}^4 \|\cormat - \iden\|_F^k,
\end{multline}
for any fixed  $0 < \gamma < 1/2$.
\end{lemma}

\begin{lemma}[Probability bound for $III$] \label{lem:ProbBddIII}
For any $0 < c < \frac{1}{2}$, there exists $C > 0 $ such that
\[
P\left(|III| > C \frac{m^2}{n^{5c}}\right) \lesssim (n^{c-1} \log m +  n^{c - 1/2} \sqrt{ \log m})
\]
for large enough $m, n$.
\end{lemma}

The proofs of \lemsref{BerryEsseen} and \lemssref{varbdd} are separately given in the next two sections. \lemref{ProbBddIII} is proved by a standard maximal inequality in \appref{PftailIII}. With these tools we can now establish \thmref{main} based on the general approach laid out in \citet{CaiMa}.

\begin{proof} [Proof of \thmref{main}]
From \eqref{ourtestform} and \eqref{TandOthers} the power of our test can be written as
\begin{equation}
\mathbb{E}[\psi]
 = P\left(I +  II   + III -  \frac{m(m-1)}{2n}> \frac{m}{n}z_{\alpha}- 2^{-1}\|{\bf R} - \iden\|_F^2\right) \label{FormOfPower}.
\end{equation} By dividing the set $\Theta(b)$ into two subsets
\[
\Theta(b, B) =   \{{\bf R}:B \sqrt{m/n} > \|{\bf R} - \iden \|_F \geq b \sqrt{m/n} \}
\]
and
\[
\Theta(B) =  \{\cormat: \|\cormat - \iden \|_F \geq B \sqrt{m/n}\},
\]
where $B$ is a sufficiently large constant depending on $(\alpha, b, \kappa)$, it suffices to show
\begin{equation} \label{Theta( B)}
\liminf_{n \rightarrow \infty} \inf_{\Theta(B)} \mathbb{E}_{\cormat} [\psi] \geq \bar{\Phi}\left(z_{\alpha} - \frac{b^2}{2}\right)
\end{equation}
and
\begin{equation} \label{Theta(b, B)}
\sup_{\Theta(b, B)}\left|\mathbb{E}_{\cormat} \psi - \bar{\Phi}\left(z_{\alpha} - \frac{\|\cormat - \iden\|_F^2}{2 m/n}\right)\right| \longrightarrow 0
\end{equation}
as $m , n \longrightarrow \infty$, $m/n \leq \kappa$. Together, they lead to the theorem since  \eqref{Theta(b, B)} implies that
\[
\lim_{n \rightarrow \infty} \inf_{\Theta(b, B)} \mathbb{E}_{\cormat} \psi =  \lim_{n \rightarrow \infty} \inf_{\Theta(b, B)} \bar{\Phi}\left(z_{\alpha} - \frac{\|\cormat - \iden\|_F^2}{2 m/n}\right) = \bar{\Phi}\left(z_{\alpha} - \frac{b^2}{2}\right).
\]

To prove \eqref{Theta( B)} we first suppose that $B$ is larger than  $\sqrt{3 z_{\alpha}}$ , and let $\delta$ be any positive constant satisfying $0 < \delta  \leq 4^{-1} z_{\alpha}$.   By definition, for any $\cormat \in \Theta(B)$,  it must be the case that $\|\cormat - \iden\|_F = \tau \sqrt{m/n}$ for some $\tau \geq B$. Together with the fact that  $mn^{-1}z_{\alpha} - 2^{-1}\|\cormat - \iden\|_F^2 \leq - \frac{m\tau^2}{n6} $ and $\delta \leq 12^{-1} \tau^2$ which are consequences of the choice of $B$, by a union bound and Chebyshev's inequality we continue from \eqref{FormOfPower} and obtain
\begin{align}
1 - \mathbb{E}[\psi]
%&\leq P \left(\left|I + II + III - \frac{m(m-1)}{2n}\right| \geq \frac{\tau^2 m}{6n}  \right)\notag\\
%&= P \left(T - \frac{m(m-1)}{2 n} \leq  \frac{m}{n}z_{1  - \lambda}\right) \notag\\
%&= P \left(\left|T - \frac{m(m-1)}{2 n} -2^{-1}\|R - I\|_F^2 \right| \geq \left| z_{1  - \lambda}(m/n) -2^{-1}\|R - I\|_F^2 \right|\right) \notag \\
&\leq P \left(\left|I + II - \frac{m(m-1)}{2n}\right| \geq \frac{\tau^2 m}{6n}  - \delta \frac{m}{n}\right)   + P\left(|III| > \delta \frac{m}{n}\right)\notag\\
%&\leq \frac{2n^2 m^{-2}\left(\mathbb{E}[I^2] + \mathbb{E}\left[\left(II - \frac{m(m-1)}{2n}
%\right)^2\right]\right)}{(6^{-1} \tau^2 - \delta)^2} +
%P\left(|III| > \delta \frac{m}{n}\right) \notag\\
&\leq 288\tau^{-4} n^2 m^{-2}\left(\mathbb{E}[I^2] + \mathbb{E}\left[\left(II - \frac{m(m-1)}{2n}
\right)^2\right]\right)+
P\left(|III| > \delta \frac{m}{n}\right).  \label{I1I2}
\end{align}
Substituting $\|\cormat - \iden\|_F$ for $\tau \sqrt{m/n}$ into the bounds for $\mathbb{E}[I^2]$ and $\mathbb{E}[(II - \frac{m(m-1)}{2n})^2]$ in \lemsref{BerryEsseen} and  \lemssref{varbdd}, it is seen that the first term in \eqref{I1I2} is bounded by a term of order
\[
\tau^{-4} + o(1) \left(\sum_{k = 0}^4 \tau^{-k}\right)
\]
%as $m, n \longrightarrow \infty$.
%\begin{multline*}
%\frac{n^2}{m^2 \tau^4}\Biggl\{ \frac{\tau^4 m^2/n^2}{n} + \frac{m^2}{n^3}\left(1 + \frac{\tau^2 m}{n}\right)  +
% \frac{m^4}{n^5} +  \left(\frac{m^3}{n^4} + \frac{1}{n^2}\right)\sqrt{\frac{m}{n}}\tau+
% \frac{\tau^2 m/n}{n}   + \\
%  \frac{m(m-1)n(n-1)}{n^4} + \frac{2(n^2 - n)}{n^4}O\left( \left\{ (1 + m)\frac{m\tau^2}{n} + \frac{\tau^4 m^2}{n^2}\right\}\right)\Biggr\}\\
%\lesssim \frac{1}{n \tau^4} + \frac{1}{n\tau^4} + \frac{ m}{\tau^2n^2}+ \frac{m^2}{\tau^4n^3} + \frac{m^{3/2}}{\tau^3n^{5/2}} + \frac{1}{\sqrt{n}\tau^3  m^{3/2}}
%+\frac{1}{m \tau^2} + \frac{O(1)}{\tau^4} + O(\frac{1}{\tau^2 n} + \frac{1}{n^2})\\
%= o(1) + \frac{O(1)}{\tau^4} \text{ as $m,n \longrightarrow \infty$},
%\end{multline*}
%where the $o(1), O(1)$ term above is uniform over $\Theta(B)$.
 Moreover, the second term  in \eqref{I1I2}
converges to $0$ as $m, n \longrightarrow \infty$ by \lemref{ProbBddIII} since $\delta m/n $ is larger than  $m^2/n^{5c}$ asymptotically for any constant $2/5 <  c < 1/2$, given that $m/n \leq \kappa$ . %for any $2/5 < c < 1/2$ and any constant $C$.
%Since $B < \tau$, this, together with the fact that the second term in \eqref{I1I2} also
%\[
%P\left(|III| < \delta \frac{m}{n}\right) \longrightarrow 0 \text{ as } m,n \longrightarrow \infty, m/n \leq \kappa
%\]
%by \lemref{ProbBddIII},
They together imply  that the constant $B= B(\alpha, b, \kappa)$ can be taken large enough so that
\[
1 - \inf_{\Theta(B)}\mathbb{E}_{\cormat}[\psi] \leq \Phi\left(z_{\alpha} - \frac{b^2}{2}\right) \text{ as } m, n \longrightarrow \infty,
\]
which is equivalent to \eqref{Theta( B)}.
%\[
%\liminf_{m, n \longrightarrow \infty} \mathbb{E}_R[\psi] \leq 1 - \Phi\left(z_{\alpha} - \frac{b^2}{2}\right).
%\]

To show \eqref{Theta(b, B)},  the uniform convergence of power on the ``stripe" of alternatives with the signal  $\|\cormat - \iden\|_F$ bounded from above and below in size, we shall first establish that
 \begin{equation}\label{Itildebdd}
P\left(|\tilde{I}| \geq \frac{m^{1- \gamma}}{n}\right) = o(1) \quad  \text{   as } \quad m, n\longrightarrow \infty \quad \text{ and }\quad  m/n \leq \kappa,
\end{equation}
uniformly over the set $\Theta(b, B)$, where
\[
\tilde{I}:= II - \frac{m(m-1)}{2n} + III.
\]
and $\gamma$ is any number such that $0 < \gamma < 1/2$.
By a union bound we have
\begin{align}
P\left(|\tilde{I}| \geq \frac{m^{1 - \gamma}}{n}\right)\notag
 &\leq P\left(\left|III\right| \geq \frac{m^{1- \gamma}}{2n} \right) + P\left( \left|II - \frac{m(m-1)}{2n}\right| \geq \frac{m^{1 - \gamma}}{2n} \right)\notag\\
& \lesssim n^{c-1} \log m +  n^{c - 1/2} \sqrt{ \log m} + \frac{n^2}{m^{2 ( 1- \gamma)}}\mathbb{E}\left[\left(II - \frac{m(m-1)}{2n}\right)^2\right] \label{bddtobeapplied}
\end{align}
for any $(2+\gamma)/5 < c < 1/2$ and large enough $m, n$. The last inequality comes from the Chebyshev inequality and the fact that, by taking $(2+\gamma)/5 < c < 1/2$ in \lemref{ProbBddIII}, for large enough $m, n$, under $m/n \leq \kappa$,  we have
\begin{align*}
P\left(|III| \geq \frac{m^{1 - \gamma}}{2n}\right) \leq P\left(|III| \geq  \frac{m^2}{2\kappa^{1+\gamma}n^{2 + \gamma}}\right)\leq P\left(|III| \geq C \frac{m^2}{n^{5c}}\right),
\end{align*}
where the constant $C$ is same as the one in \lemref{ProbBddIII}. Since $\cormat \in \Theta(b, B)$, it must be that $\|\cormat - \iden\|_F = \tau \sqrt{m/n}$ for some $b \leq \tau \leq B$, and substituting this into the variance bound in \lemref{varbdd} it can be easily seen that
\begin{equation} \label{perfect1}
\frac{n^2}{m^{2 ( 1- \gamma)}}\mathbb{E}\left[\left(II - \frac{m(m-1)}{2n}\right)^2\right] \longrightarrow 0
\end{equation}
uniformly over $\Theta(b, B)$ as $m, n \longrightarrow \infty$, $m/n \leq \kappa$.  This gives \eqref{Itildebdd}  since $c < 1/2$ in \eqref{bddtobeapplied}.

 To finish the proof of  \eqref{Theta(b, B)},   by union bound arguments one has
\begin{equation*}
\mathbb{E}[\psi]\leq P\left( I \geq  \frac{mz_{\alpha}}{n}- \frac{\|\cormat - \iden\|_F^2}{2} - \frac{m^{1 - \gamma }}{ n}\right) +
 P\left( |\tilde{I}| \geq \frac{m^{1 - \gamma}}{n} \right)
\end{equation*}
and
\begin{equation*}
\mathbb{E}[\psi] \geq
P\left( I \geq  \frac{mz_{\alpha}}{n}- \frac{\|\cormat - \iden\|_F^2}{2} + \frac{m^{1 - \gamma}}{n}\right) -  P\left( |\tilde{I}| \geq \frac{m^{1 - \gamma}}{n} \right),
\end{equation*}
which collectively imply
\begin{multline} \label{combiningUnions}
\left|\mathbb{E}[\psi] - \bar{\Phi}\left( \frac{mz_{\alpha}n^{-1}- 2^{-1}\|\cormat - \iden\|_F^2}{\sqrt{\text{Var}(I)}}\right)\right|\\
\leq \sup_{t \in \mathbb{R}}\left|P\left( \frac{I}{\sqrt{\text{Var}(I)}} \geq t \right) -
 \bar{\Phi}\left(t\right) \right|
 +   2P\left( |\tilde{I}| \geq \frac{m^{1 - \gamma}}{n} \right)+ \frac{2m^{1- \gamma}n^{-1}}{\sqrt{\text{Var}(I)}}
\end{multline}
since $|\bar{\Phi}(x \pm \epsilon) - \bar{\Phi}(x)| \leq \epsilon$ for any $x \in \mathbb{R}$ and $\epsilon \geq 0$.  Moreover, all three terms on the right hand side of \eqref{combiningUnions} are of order $o(1)$ uniformly over $\Theta(b, B)$. The first two terms are so by \lemref{BerryEsseen}$(ii)$ and \eqref{Itildebdd}, and the last term is so since by \lemref{BerryEsseen}$(i)$, $\sqrt{\text{Var}(I)} = m/n + o(m^{1 - \gamma}/n)$ where the $o(m^{1 - \gamma}/n)$ term is also uniform over $\Theta(b, B)$.
%
%
%, which gives
%\begin{multline} \label{firsto1}
%\sup_{\Theta(b, B)}\left|\mathbb{E}_R[\psi] - \bar{\Phi}\left( \frac{ m z_{1 - \lambda}n^{-1}-2^{-1}\|{\bf R} - {\bf I}_m\|_F^2}{\sqrt{\text{Var}(I)}}\right)\right| \leq\\
% \sup_{t \in \mathbb{R}}\left|P\left( \frac{I}{\sqrt{\text{Var}(I)}} \geq t \right) -
%\bar{\Phi}\left(t\right) \right| + o(1)
%\end{multline}
%together with  \eqref{Itildebdd}.
Finally, by  \lemref{BerryEsseen}$(i)$ as $m, n \longrightarrow \infty$, $m/n \leq \kappa$, we also have
\begin{equation} \label{perfect2}
\sup_{\Theta(b, B)} \left| \frac{\text{Var}(I)}{m^2/n^2} - 1 \right| \longrightarrow 0,
\end{equation}
and it is not hard to see that this implies
\begin{equation*}\label{secondo1}
\sup_{\Theta(b, B)}\left|\bar{\Phi}\left(z_{\alpha} - \frac{\|\cormat - \iden\|_F^2}{2 m/n}\right) -\bar{\Phi}\left( \frac{m z_{\alpha}n^{-1}-2^{-1}\|\cormat - \iden\|_F^2}{\sqrt{\text{Var}(I)}}\right)\right| \longrightarrow 0.
\end{equation*}
 Applying these facts to \eqref{combiningUnions} leads to \eqref{Theta(b, B)}.
\end{proof}

 In establishing the normal tail form of our power function,  perhaps the most important step is singling out $I$ as the main term that drives the asymptotic normality of the left hand side in \eqref{TandOthers}  under the ``stripe" of alternative $\Theta(b, B)$ via the Berry-Esseen bound in \lemref{BerryEsseen}$(ii)$.  We note that $I$ is already a rather simple term to handle, but proving \lemref{BerryEsseen}$(ii)$ for it still takes considerable effort in the next section.
% Moreover, we didn't intend to present the sharpest possible bounds in our three building lemmas: \lemsref{BerryEsseen}, \lemssref{varbdd} and \lemssref{ProbBddIII}; they are nevertheless good enough for our purpose.
Moreover,  $m/n \leq \kappa$ has been used at different places, the convergences in \eqref{perfect1} and \eqref{perfect2} for instances. However, the assumption is mostly a convenient one for such statements regarding terms $I$ and $II$,  since the estimates presented in \lemsref{BerryEsseen} and \lemssref{varbdd} are not the sharpest possible, for either aesthetic purpose or saving us some  effort on refining them in the next two sections.

 It is the remainder term $III$ that truly prevents us from removing the upper bound on $m/n$. In order to show it tends to zero in probability, as in \eqref{bddtobeapplied}, we applied the crude tail bound in \lemref{ProbBddIII} based on a maximal inequality (see \appref{PftailIII}). Such an estimate doesn't take the correlations among the constituent summands $III_{pq}$ into account,  as was done for the $II_{pq}$'s with respective to $II- (m-1)m(2n)^{-1}$  via explicitly estimating its second moment in \lemref{varbdd}. The major obstacle to computing $\mathbb{E}[III^2]$ is the  {\it random} coefficients
\begin{equation} \label{trouble}
\frac{(\rho_{pq} + k_{pq} \bar{S}_{pq})^{2 - \lambda_1}}{{(1 + k_{pq}\bar{S}_{pp})}^{1 + \lambda_2 } {( 1 +k_{pq}\bar{S}_{qq} )}^{1 + \lambda_3 }}
\end{equation}
attached to the products $\bar{S}_{pp}^{\lambda_1} \bar{S}_{qq}^{\lambda_2}\bar{S}_{pq}^{\lambda_3}$  in definition \eqref{remainder}. Unlike $II$, where the constituents $II_{pq}$ have {\it constant} coefficients, not only is the coefficient in \eqref{trouble}  a {\it rational} functions in $\bar{S}_{pp}$, $\bar{S}_{pq}$, $\bar{S}_{qq}$, but it also involves the intractable random quantity $k_{pq} = k_{pq}(\bar{S}_{pp}, \bar{S}_{pq}, \bar{S}_{qq}) \in (0, 1)$. As such, there is no straightforward way of applying Isserlis's theorem (\thmref{isserlis}) to compute the moment $\mathbb{E}[III^2]$ like we did for $\mathbb{E}[(II- (m-1)m(2n)^{-1})^2]$ in \secref{varbddsec}. In fact, even with the help of \texttt{mathematica}, it still took us substantial effort to get our bound in \lemref{varbdd} as seen later. At this moment, we cannot think of other ways to control term $III$.

\section{The Berry Esseen bound for $I$} \label{sec:BEsec}

We will prove \lemref{BerryEsseen} in this section. For our presentation, given a finite set $D$  and $|D|$ duples $(p_d, q_d) \in [m] \times [ m ]$ indexed by a subscript $d$ that ranges over $D$, we define the central moment quantities
\[
\mathcal{M}_{\substack{(p_d, q_d)\\ d \in D}} := \mathbb{E}\left[ \prod_{d \in D}(X_{p_d}X_{q_d} - \rho_{p_dq_d})\right].
 \]

Recall that $I$ is defined as $\sum_{p <q }I_{pq}$, where each $I_{pq}$ is given in \eqref{I_pq}. We first observe that  $I$ has a natural martingale structure: For each $i = 1, \dots, n$,  let $\mathcal{F}_i$ be the sigma-algebra generated by $\{X_{p j}: 1\leq p \leq m; 1 \leq j \leq i\}$ and $\mathcal{F}_0$ be the trivial sigma algebra, and define
\begin{equation} \label{Yidef}
Y_i := \frac{2}{n^2}\sum_{p < q} \sum_{j < i}(X_{pi}X_{qi} - \rho_{pq})(X_{pj}X_{qj} - \rho_{pq}) \text{ for } i = 2, \dots, n
\end{equation}
as well as
\begin{equation}\label{Y01def}
Y_0 = Y_1 := 0.
\end{equation}
Then $I = \sum_{i = 0}^nY_i$, and $(Y_i)_{i =0}^n$ is a the sequence of martingale differences   since
\[
\mathbb{E}[Y_i|\mathcal{F}_{i-1}] =  \sum_{p < q} \frac{2}{n^2}\sum_{j < i}(X_{pj}X_{qj} - \rho_{pq})\mathbb{E}[X_{pi}X_{qi} - \rho_{pq}] = 0
\] for $i \geq 2$, where $\mathbb{E}[Y_i|\mathcal{F}_{i-1}] = 0$ is trivial for $i = 0, 1$.

With the observations just made it is easy to see that  $\mathbb{E}[I] = 0$ and

\begin{equation} \label{sumYi2}
\text{Var}[I] = \mathbb{E}[I^2] = \sum_{i = 2}^n \mathbb{E}[Y_i^2].
\end{equation}
By the i.i.d.'ness of the samples,  for each $i = 2, \dots, n$,
\begin{align}
\mathbb{E}[Y_i^2] &= \frac{4}{n^4} \sum_{\substack{1 \leq p_d < q_d \leq m\\ d= 1, 2 }} \mathcal{M}_{\substack{(p_d, q_d)\\d \in[2] }} \left(\sum_{1  \leq j, j' < i}\mathbb{E}\left[(X_{p_1j'}X_{q_1j'} - \rho_{p_1q_1})(X_{p_2j}X_{q_2j} - \rho_{p_2q_2})\right] \right)
 \notag\\
 &= \frac{4(i-1)}{n^4}\sum_{\substack{1 \leq p_d < q_d \leq m\\ d= 1, 2 }} \mathcal{M}_{\substack{(p_d, q_d)\\d \in [2] }}^2
 \label{innermost},
\end{align}
where, to clarify, $\sum_{\substack{1 \leq p_d < q_d \leq m\\d= 1, 2}}$ means a summation over all pairs of duples $\{(p_1 ,  q_1 ), (p_2, q_2)\}$ such that $1 \leq p_d < q_d \leq m$ for each $d = 1, 2$. We have the equality in \eqref{innermost} because $\mathbb{E}[(X_{p_1j'}X_{q_1j'} - \rho_{p_1q_1})(X_{p_2j}X_{q_2j} - \rho_{p_2q_2})]$ equals $\mathcal{M}_{\substack{(p_d, q_d)\\d \in \{1, 2\} }}$ when $j = j'$ and zero otherwise.  For $k = 2, 3, 4$, let
\begin{equation} \label{DefSk}
\mathbb{S}(k) := \sum_{\substack{1 \leq p_d < q_d \leq m\\ d = 1, 2\\|\cup_{d  = 1}^2\{p_d, q_d\}| = k}} \mathcal{M}_{\substack{(p_d, q_d)\\d \in [2] }}^2
\end{equation}
correspond to a sum over all duples $1 \leq p_d < q_d \leq m$, $d = 1, 2$ such that as a set $\cup_{d  = 1}^2\{p_d, q_d\}$ has cardinality $k$. From \eqref{sumYi2} and \eqref{innermost} we can write
\begin{equation} \label{VarIForm}
\text{Var}[I] = \frac{2n(n-1)}{n^4}\sum_{k = 2}^4 \mathbb{S}(k).
\end{equation}
since $\sum_{i = 2}^n (i-1) = 2^{-1}(n^2 - n)$.
In \appref{bddforBs}, we will show the following estimates hold:
\begin{align}
\mathbb{S}(2)  &= 2^{-1}m(m-1) + O(\|\cormat - \iden\|_F^2) \label{S2est}
\\
\mathbb{S}(3)  &= O( m \|\cormat - \iden\|_F^2 + \|\cormat - \iden\|_F^4)  \label{S3est}
\\
\mathbb{S}(4)  &=O(  \|\cormat - \iden\|_F^4)  \label{S4est}
\end{align}
Substituting these into \eqref{VarIForm} results in \lemref{BerryEsseen}$(i)$. In fact, this general strategy of decomposing a sum according to the cardinality of an index set as in \eqref{DefSk} and forming separate estimates will be employed repeatedly in the sequel.

%With these three bounds we get that
%\[
%\mathbb{E}[Y_i^2] = \frac{4(i-1)}{n^4}\left\{\frac{m(m-1)}{2} +O(m\|R  - I\|_F^2 +\|R  - I\|_F^4 )  \right\} \text{ for } i = 2, \dots n.
%\]
%Given that $\sum_{i = 2}^n (i-1) = \frac{n^2 - n}{2}$, we have proved \eqref{bddVarI}.

We shall now prove the normal approximation in \lemref{BerryEsseen}$(ii)$.  With a  Berry-Esseen theorem for martingale central limit theorem in \citet{BrownHeyde},  it suffices to verify the fourth moment conditions
\begin{equation} \label{BE1}
\sum_{i = 2}^n \mathbb{E}[Y_i^4] =  o(m^4/n^4) \sum_{k =0}^4 \|\cormat - \iden\|_F^k
\end{equation}
and
 \begin{align}
\mathbb{E}\left[\left(\sum_{i = 2}^n \mathbb{E}[Y_i^2 |\mathcal{F}_{i-1}]- \text{Var}(I)\right)^2 \right]  &=   \mathbb{E}\left[\left(\sum_{i = 2}^n \mathbb{E}[Y_i^2 |\mathcal{F}_{i-1}]\right)^2\right]- \text{Var}(I)^2\notag\\
&= o(m^4/n^4) \sum_{k = 0}^8 \|\cormat - \iden\|_F^k
%&= O\left(\frac{m^4}{n^5} \sum_{k = 0}^4 \|\cormat - \iden\|^k_F + \frac{m^3}{n^4}+ \frac{\sum_{k = 0}^3 m^k\|\cormat - \iden\|_F^{2(4-k)} +  m^3\|\cormat - \iden\|_F}{n^4}\right)
 \label{BE2}.
\end{align}
Note that the equality before \eqref{BE2} holds because $\mathbb{E}[\sum_{i = 2}^n \mathbb{E}[Y_i^2 |\mathcal{F}_{i-1}]] = \mathbb{E}[\sum_{i = 2}^nY_i^2] = \text{Var}(I)$.

We will first show \eqref{BE1}.  For any $2 \leq i \leq n$, on raising $Y_i$ to the $4$th power and taking expectation, by the i.i.d.'ness of samples,  we have
\begin{align}
&\mathbb{E}[Y_i^4] \notag\\
&= \frac{16}{n^8} \sum_{\substack{1 \leq p_d < q_d \leq m\\d= 1, 2, 3, 4}}  \left\{\mathbb{E}\left[\prod_{d =1}^4 (X_{p_d i}X_{q_d i}- \rho_{p_d q_d})\right] \sum_{ \substack{1 \leq j_d< i\\ d= 1, 2, 3, 4}} \mathbb{E}\left[ \prod_{d = 1}^4  ( X_{ p_d j_d} X_{ q_d j_d}- \rho_{p_d q_d}) \right]\right\}\notag\\
&= \frac{16}{n^8} \sum_{\substack{1 \leq p_d < q_d \leq m\\d= 1, 2, 3, 4}}  \left\{ \mathcal{M}_{\substack{ (p_d, q_d)\\ d \in [4]}} \sum_{ \substack{1 \leq j_d< i\\ d= 1, 2, 3, 4}} \mathbb{E}\left[ \prod_{d = 1}^4  ( X_{ p_d j_d} X_{ q_d j_d}- \rho_{p_d q_d}) \right]\right\}\notag\\
&=  O\left( \frac{ i^2}{n^8} \right) \sum_{\substack{1 \leq p_d < q_d \leq m\\d= 1, 2, 3, 4}}  \mathcal{M}_{\substack{ (p_d, q_d)\\ d \in [4]}} \label{EY4} ,
\end{align}
where the summations $\sum_{\substack{1 \leq p_d < q_d \leq m\\d= 1, 2, 3, 4}}$ and $\sum_{ \substack{1 \leq j_d< i\\ d= 1, 2, 3, 4}}$ are defined similarly as the one in \eqref{innermost}. The last equality in \eqref{EY4} is explained as follows: For a fixed $i$ and a given set of variables index pairs $\{(p_d ,  q_d ): d =1, \dots, 4\}$, with any choice of the sample indices $j_1, \dots, j_4$ in order for the expectation
\begin{equation} \label{sumofsummands}
%\sum_{ \substack{1 \leq j_d< i\\ d= 1, 2, 3, 4}}
\mathbb{E}\left[ \prod_{d = 1}^4  ( X_{ p_d j_d} X_{ q_d j_d}- \rho_{p_d q_d}) \right]
\end{equation}
to be non-zero, by independence  it must be true that  there exists a permutation function $\pi \in \mathcal{S}_4$ so that
\begin{equation}\label{nonzerocondition}
j_{\pi(1)} = j_{\pi(2)}, \ \ j_{\pi(3)} = j_{\pi(4)}.
\end{equation}
%Otherwise, there would be a $d' \in \{1, 2, 3, 4\}$ so that $j_{d'} \not = j_{d}$ with all other $d \not= d'$ in $\{1, 2, 3, 4\}$, and by independence of the samples
%\begin{multline}
%\mathbb{E}\left[ \prod_{d = 1}^4  ( X_{ p_d j_d} X_{ q_d j_d}- \rho_{p_d q_d}) \right] = \\
%\mathbb{E}\left[X_{p_{d'}j_{d'}}X_{q_{d'}j_{d'}} - \rho_{p_{d'}q_{d'}}\right] \mathbb{E}\left[\prod_{\substack{d = 1, 2, 3, 4 \\ d \not = d'}}  ( X_{ p_d j_d} X_{ q_d j_d}- \rho_{p_d q_d})\right] = 0.
%\end{multline}
Since the condition in \eqref{nonzerocondition} implies that $|\cup_{d = 1}^4\{j_d\}  | \leq 2$,  at most $O({i-1 \choose 2}) = O(i^2)$ many expectations in  \eqref{sumofsummands} can be non-zero.  This leads to \eqref{EY4} since the expectations in \eqref{sumofsummands}, when they are non-zero, can be uniformly bounded regardless of the choice for $\{(p_d, q_d, j_d); d =1, \dots, 4\}$, owing to our assumptions at the beginning of \secref{main} and \thmref{isserlis} on higher order normal moments.
Provided that $\sum_{i = 2}^n i^2 = 6^{-1} (2n^3 + 3 n^2 + n - 6)$, with \eqref{EY4} we further write
\begin{equation}\label{sumofexpY4}
 \sum_{i = 2}^n \mathbb{E}[Y_i^4] = O(n^{-5})\sum_{\substack{1 \leq p_d < q_d \leq m\\d= 1, 2, 3, 4}}  \mathcal{M}_{\substack{ (p_d, q_d)\\ d \in [4]}}.
\end{equation}
Now  the last term in \eqref{sumofexpY4}  can be decomposed, according to the cardinality of the set of duples $\cup_{d = 1}^4 \{p_d, q_d\}$, as
%For simplicity of notations, define
%\begin{equation}
%A_{p_1, \dots, p_8}:= \mathbb{E}\left[\prod_{j = 1, 3, 5, 7} (X_{p_d, 1}X_{p_{d+1}, 1}- \rho_{p_d p_{d+1}})\right]
%\end{equation}
% for any $1\leq p_1, \dots, p_8 \leq m$ and observe that the value of $A_{p_1, \dots, p_8}$ fully depends on the choice of $p_1, \dots, p_8$.
\begin{equation} \label{sumabove}
\sum_{\substack{1 \leq p_d < q_d \leq m\\d= 1, 2, 3, 4}}  \mathcal{M}_{\substack{ (p_d, q_d)\\ d \in [4]}} =
 \sum_{k = 5}^8 \mathbb{T}(k)+ O(m^4),
\end{equation}
where for $k = 2, \dots, 8$,
\[
\mathbb{T}(k) := \sum_{\substack{1 \leq p_d < q_d \leq m\\d= 1, 2, 3, 4 \\ |\cup_{d = 1}^4 \{p_d, q_d\}| = k}}   \mathcal{M}_{\substack{ (p_d, q_d)\\ d \in [4]}}
\]
%each $B_k$, $k = 5, \dots, 8$, is defined as
%\begin{equation} \label{Bk}
%B_k := \sum_{\substack{1 \leq p_d < q_d \leq m\\d= 1, 2, 3, 4 \\ |\cup_{d = 1}^4 \{p_d, q_d\}| = k}}  \mathbb{E}\left[\prod_{d =1}^4 (X_{p_d }X_{q_d }- \rho_{p_d q_d})\right]
%\end{equation}
%and
and the $O(m^4)$ term  comes from the fact that there are only $O(m^4)$  many uniformly bounded extra summands under the restriction  $|\cup_{k = 1}^4 \{p_d, q_d\}| \leq 4$.
%
% in
%\[
%\sum_{k  =2}^4 \mathbb{T}(k) = \sum_{\substack{1 \leq p_d < q_d \leq m\\d= 1, 2, 3, 4 \\ |\cup_{d = 1}^4 \{p_d, q_d\}| \leq 4}}  \mathcal{M}_{\substack{ (p_d, q_d)\\ d \in [4]}}
%\]
% and all $ \mathcal{M}_{\substack{ (p_d, q_d)\\ d \in [4]}}$  can be uniformly bounded.
In \appref{bddforBs} we will show that
\begin{equation}\label{afterIsserlis}
 \mathbb{T}(k) = O(m^4) \|\cormat - \iden\|_F^{k - 4}
\end{equation}
for each $k = 5, \dots, 8$. Collecting \eqref{sumofexpY4}, \eqref{sumabove} and \eqref{afterIsserlis} we get \eqref{BE1}.

%By \thmref{isserlis}, it is easy to see that $ \mathbb{E}\left[\prod_{d =1}^4 (X_{p_d }X_{q_d }- \rho_{p_d q_d})\right]$ is a sum of finitely many terms each having the form
%\[
%\prod_{d = 1}^4\rho_{\pi(p_d) \pi(q_d)}
%\]
%for some $\pi \in \mathcal{S}_8$, hence on applying \lemref{bddOnSums}$(i)$ we have
%\begin{equation} \label{afterIsserlis}
% \sum_{k = 5}^8 \sum_{\substack{1 \leq p_d < q_d \leq m\\d= 1, 2, 3, 4 \\ |\cup_{d = 1}^4 \{p_d, q_d\}| = k}}  \mathbb{E}\left[\prod_{d =1}^4 (X_{p_d }X_{q_d }- \rho_{p_d q_d})\right] \lesssim \sum_{k = 5}^8O(m^4) \|R - I\|_F^{k - 4}.
%\end{equation}

To show \eqref{BE2} it suffices to understand the term
$\mathbb{E}[(\sum_{i = 1}^n \mathbb{E}[Y_i^2 |\mathcal{F}_{i-1}])^2]
$ since  the form of $\text{Var}(I)$ has been proven in \lemref{BerryEsseen}$(i)$. On expansion,
\begin{multline}\label{sumcondY2}
\sum_{i = 2}^n \mathbb{E}[Y_i^2 | \mathcal{F}_{i-1}] = \\
\frac{4}{n^4}  \sum_{i = 2}^n\sum_{\substack{1\leq p_d < q_d \leq m \\d = 1, 2 }} \mathcal{M}_{\substack{(p_d, q_d)\\ d \in [2] }}\left[ \sum_{1 \leq j, k < i}(X_{p_1j}X_{q_1j} - \rho_{p_1q_1})(X_{p_2k}X_{q_2k} - \rho_{p_2 q_2})\right].
\end{multline}
 Proceeding with our calculations,
%\begin{equation}  \label{generatefact}
%\mathbb{E}[ (X_p X_q - \rho_{pq}) (X_r X_s - \rho_{rs})] = \rho_{ps} \rho_{qr} + \rho_{pr} \rho_{qs},
%\end{equation}
%for any $1\leq p,q, r, s \leq m$. Hence,

\begin{multline} \label{ExpsumcondY22}
 \mathbb{E}\left[\left(\sum_{i = 2}^n \mathbb{E}[Y_i^2 |\mathcal{F}_{i-1}]\right)^2 \right]= \\
\frac{16}{n^8} \sum_{\substack{1\leq p_d < q_d \leq m \\d = 1, 2, 3, 4}}\Biggl\{ \mathbb{P}_1 \times
\sum_{2 \leq i, j \leq n} \sum_{\substack{1 \leq  i_1, i_2 < i \\1 \leq  i_3, i_4 < j}}\mathbb{E}\left[ \prod_{d = 1}^4 (X_{p_d i_d}X_{q_d i_d} - \rho_{p_dq_d})  \right] \Biggr\},
\end{multline}
where
\begin{align} \label{P1}
\mathbb{P}_1= \mathbb{P}_1 (p_1, q_1, \dots,p_4, q_4 ) &:= \mathcal{M}_{\substack{(p_d, q_d)\\ d \in \{1, 2\} }} \mathcal{M}_{\substack{(p_d, q_d)\\ d \in \{3, 4\} }}
\end{align}
By independence, we note that the expression
\[
\mathbb{E}\left[ \prod_{d = 1}^4 (X_{p_d i_d}X_{q_d i_d} - \rho_{p_dq_d})\right]
\]
on the right hand side of \eqref{ExpsumcondY22} can be non-zero only if  the four sample indices $i_1, \dots, i_4$  are such that either
\begin{equation} \label{1criterion}
 i_1 = \cdots = i_4,
\end{equation}
\begin{equation} \label{2criterion}
 i_1 = i_2, \quad i_3 = i_4, \quad  |\{i_1, \dots, i_4\}| = 2,
\end{equation}
\begin{equation} \label{3criterion}
 i_1 = i_3, \quad i_2 = i_4, \quad  |\{i_1, \dots, i_4\}| = 2
\end{equation}
or
 \begin{equation}\label{4criterion}
 i_1 = i_4, \quad i_2 = i_3, \quad  |\{i_1, \dots, i_4\}| = 2.
\end{equation}
For any fixed given pair $2 \leq i ,  j \leq n$, by simple counting, there are, respectively,
$ i \wedge j-1$,  $(i \wedge j-1)(i \vee j-2)$, $(i \wedge j-1)(i \wedge j-2)$, $(i \wedge j-1)(i \wedge j-2)$
combinations of $(i_1, i_2, i_3, i_4)$ that satisfy \eqref{1criterion},  \eqref{2criterion},  \eqref{3criterion},  \eqref{4criterion}  for which $1 \leq i_1, i_2 < i$ and $1 \leq i_3, i_4 < j $, where $a \vee b = \max(a, b)$  and $a \wedge b = \min(a, b)$ .
Hence,
\begin{align}
&\sum_{2 \leq i, j \leq n} \sum_{\substack{1 \leq  i_1, i_2 < i \\1 \leq  i_3, i_4 < j}}\mathbb{E}\left[ \prod_{d = 1}^4 (X_{p_d i_d}X_{q_d i_d} - \rho_{p_dq_d})  \right]
\notag\\
& = \mathcal{M}_{\substack{  (p_d, q_d)\\ d \in [4]}} \underbrace{ \left\{ \sum_{2 \leq i, j \leq n } (i \wedge j  - 1) \right\} }_{  = 6^{-1} (2n^3 - 3n^2 + n)}+
 \mathbb{P}_1\underbrace{ \sum_{2 \leq i, j \leq n}(i\wedge j-1)(i\vee j-2)}_{ = 12^{-1} (-2 n + 9 n^2 - 10 n^3 + 3 n^4)}\notag\\
& \qquad + (\mathbb{P}_2 + \mathbb{P}_3)\underbrace{\left\{ \sum_{2 \leq i, j \leq n }(i \wedge j-1)(i \wedge j-2) \right\}}_{  = 6^{-1} (n^4 - 4n^3 + 5n^2 - 2n)}\notag\\
&=\mathcal{M}_{\substack{  (p_d, q_d)\\ d \in [4]}} O(n^3) + \mathbb{P}_1 \left(\frac{n^4 }{4} + O(n^3)\right)+ \left(\mathbb{P}_2 + \mathbb{P}_3\right) O(n^4) \label{tobesub},
\end{align}
where
\begin{align*}
%\mathbb{P}_2 = \mathbb{P}_2(p_1, q_1, \dots, p_4, q_4) :=  \mathcal{M}_{\substack{  (p_d, q_d)\\ d \in \{1, 2\}}}\mathcal{M}_{\substack{  (p_d, q_d)\\ d \in \{3, 4\}}} \label{P2}\\
\mathbb{P}_2= \mathbb{P}_2(p_1, q_1, \dots, p_4, q_4)  := \mathcal{M}_{\substack{  (p_d, q_d)\\ d \in \{1, 3\}}}\mathcal{M}_{\substack{  (p_d, q_d)\\ d \in \{2, 4\}}}\\
\mathbb{P}_3= \mathbb{P}_3(p_1, q_1, \dots, p_4, q_4)   := \mathcal{M}_{\substack{  (p_d, q_d)\\ d \in \{1, 4\}}}\mathcal{M}_{\substack{  (p_d, q_d)\\ d \in \{2, 3\}}}
\end{align*}
are the value of $\mathbb{E}[ \prod_{d = 1}^4 (X_{p_d i_d}X_{q_d i_d} - \rho_{p_dq_d})] $ when $i_1, \dots, i_4$ satisfy the criteria   \eqref{3criterion} and  \eqref{4criterion} respectively.   Substituting \eqref{tobesub} into \eqref{ExpsumcondY22} gives

\begin{multline} \label{aftersub}
 \mathbb{E}\left[\left(\sum_{i = 2}^n \mathbb{E}[Y_i^2 |\mathcal{F}_{i-1}]\right)^2 \right]= \\
O(n^{-5}) \sum_{\substack{1\leq p_d < q_d \leq m \\d = 1, 2, 3, 4}} \mathbb{P}_1  + \left(\frac{4}{n^4} + O(n^{-5})\right) \sum_{\substack{1\leq p_d < q_d \leq m \\d = 1, 2, 3, 4}} \mathbb{P}_1^2+ O(n^{-4})\sum_{\substack{1\leq p_d < q_d \leq m \\d = 1, 2, 3, 4}} \sum_{u = 2}^3 \mathbb{P}_1\mathbb{P}_u ,
\end{multline}
where  the terms $\mathcal{M}_{\substack{  (p_d, q_d)\\ d \in [4]}}$ in\eqref{tobesub} are absorbed into the first $O(n^{-5})$ term  because
they are  uniformly bounded regardless of the choice of $p_1, q_1, \dots, p_4, q_4$, again by our assumptions and \thmref{isserlis}. From this it remains to show the estimates
\begin{equation} \label{firstBrace}
 \sum_{\substack{1\leq p_d < q_d \leq m \\d = 1, 2, 3, 4}}\mathbb{P}_1 =  O(m^4) \sum_{k = 0}^4 \|\cormat - \iden\|_F^k,
\end{equation}

\begin{equation}\label{secondBrace}
\sum_{\substack{1\leq p_d < q_d \leq m \\d = 1, 2, 3, 4}}\mathbb{P}_1^2 = \frac{m^4}{4}  + O(m^3) \sum_{k = 0}^8\|\cormat - \iden\|_F^k ,
% \sum_{k = 0}^3 O(m^k)\|\cormat - \iden\|_F^{2(4-k)} +  O(m^3)\|\cormat - \iden\|_F + O(m^3),
\end{equation}
and
\begin{equation} \label{thirdBrace}
\sum_{\substack{1\leq p_d < q_d \leq m \\d = 1, 2, 3, 4}} \sum_{u = 2}^3  \mathbb{P}_1\mathbb{P}_u =O(m^3) \sum_{k = 0}^8\|\cormat - \iden\|_F^k,
% \sum_{k = 0}^3 O(m^k)\|\cormat - \iden\|_F^{2(4-k)} + O(m^3)\|\cormat - \iden\|_F + O(m^3),
\end{equation}
which, together with  \lemref{BerryEsseen}$(i)$ and \eqref{aftersub}, imply \eqref{BE2}. The proofs of these  estimates will, again, be deferred to \appref{bddforBs}.

\section{The second moment bound for $II - \frac{m(m-1)}{2n}$} \label{sec:varbddsec}

We will now prove \lemref{varbdd}. Recall that $II :=\sum_{p < q } II_{pq}$ , and from the definition of $II_{pq}$ in \eqref{II_pq} we can equivalently write it as
\[
II_{pq} = {II}_{pq, 1} + II_{pq, 2},
\]
where
\begin{multline}\label{II_pq1}
II_{pq, 1}  :=    \frac{\sum_{i = 1}^n (X_{pi}X_{qi} - \rho_{pq})^2}{n^2}+
\sum_{\substack{{\boldsymbol \lambda} \in \mathbb{N}^3_{\geq 0}:  \\3 \leq |{\boldsymbol \lambda}| \leq 4  \\ \lambda_3 = 2 \\ {\boldsymbol \lambda \not = (1, 1, 2)}}} \frac{\partial^{\boldsymbol \lambda} f(1, 1, \rho_{pq})}{{\boldsymbol \lambda}!}\bar{S}_{pp}^{\lambda_1} \bar{S}_{qq}^{\lambda_2} \bar{S}_{pq}^{\lambda_3}
\end{multline}
and
\begin{multline} \label{II_pq2}
II_{pq, 2}  := \frac{\partial^{(1, 1, 2)}f(1, 1, \rho_{pq})}{1!1!2!}\bar{S}_{pp} \bar{S}_{qq} \bar{S}_{pq}^2 +
 \sum_{\substack{{\boldsymbol \lambda} \in \mathbb{N}^3_{\geq 0}:  \\1 \leq |{\boldsymbol \lambda}| \leq 4  \\ \lambda_3 \not= 2}} \frac{\partial^{\boldsymbol \lambda} f(1, 1, \rho_{pq})}{{\boldsymbol \lambda}!}\bar{S}_{pp}^{\lambda_1} \bar{S}_{qq}^{\lambda_2} \bar{S}_{pq}^{\lambda_3} .
%\\+  \ \ \sum_{p < q }\frac{\partial^{(1, 1, 2)} f(1, 1, \rho_{pq})}{1!1!2!}(S_{pp} - 1) (S_{qq} - 1)(S_{pq} - \rho_{pq})^2.
\end{multline}
%by grouping terms based on whether the third derivative index $\lambda_3$ equal to $2$ or not.
We form this grouping of terms for reasons that will be explained later.
As such, by defining $II_1 := \sum_{p < q}II_{pq, 1}$ and  $II_2 := \sum_{p < q}II_{pq, 2}$, one can write
\[
II =II_1 +II_2.
\]
To finish the proof of \lemref{varbdd}, it suffices to bound  the second moments of $II_1- \frac{m(m-1)}{2n}$  and $II_2$ respectively in terms of $\|\cormat- \iden\|_F$.

% to estabish since for all $II_{pq, 2}$, $1 \leq p<q \leq m $, the coefficients $\partial^{\boldsymbol \lambda} f(1, 1, \rho_{pq})$ in \eqref{II_pq} are of the form  $C \rho_{pq}$ or $C \rho_{pq}^2$  for some constant $C \not = 0$, as suggested by \lemref{df}.

\begin{lemma}[Bound on the second moment of $ II_1 - \frac{m(m-1)}{2n}$] \label{lem:II_1bdd}
%\begin{multline*}
%\mathbb{E}\left[\left(II_1- \frac{m(m-1)}{2n}\right)^2\right] \lesssim\\
%\frac{\|\cormat  - \iden\|^4_F }{n^2} + \frac{m^2}{n^3}\left(1 + \|\cormat  - \iden\|_F^2\right)  +
% \frac{m^4}{n^5} +  \left(\frac{m^3}{n^4} + \frac{1}{n^2}\right)\|\cormat  - \iden\|_F
%\end{multline*}
\begin{equation*}
\mathbb{E}\left[\left(II_1- \frac{m(m-1)}{2n}\right)^2\right] \lesssim
%\frac{\|\cormat  - \iden\|^4_F }{n^2} +
o\left(\frac{m^{2 (1 - \gamma)}}{n^2}\right) \sum_{k = 0}^4 \|\cormat - I\|_F^k
\end{equation*}
for any $0 < \gamma < 1/2$.
\end{lemma}

\begin{lemma}[Bound on the second moment of $ II_2 $] \label{lem:II_2bdd}
 \begin{equation}\label{bddinII_2bdd}
\mathbb{E} \left[\left(II_2 \right)^2\right] \lesssim  \frac{\|\cormat  - \iden\|_F^2 + \|\cormat  - \iden\|_F^4}{n}
%+ \frac{m \|R  - I\|_F^{3/2}}{n^2}
%+ \frac{m^2 \|\cormat  - \iden\|_F^2}{n^3}
+ o\left(\frac{m^{2 (1 - \gamma)}}{n^2}\right) \sum_{k = 0}^2 \|\cormat - I\|_F^k
\end{equation}
for any $0 < \gamma < 1/2$.
\end{lemma}

 Using \lemsref{II_1bdd} and \lemssref{II_2bdd}, \lemref{varbdd} immediately follows from \begin{inparaenum} \item $II^2  = (II_1 - \frac{m(m-1)}{2})^2 +  II_2^2  + 2   (II_1 - \frac{m(m-1)}{2})II_2$ and  \item $ 2| (II_1 - \frac{m(m-1)}{2}) II_2| \leq  (II_1 - \frac{m(m-1)}{2})^2 +  II_2^2 $\end{inparaenum} .

%With \lemref{df}, we explicitly write
%\begin{align}
%&II_2 \\leq 2 \rho_{pq}\Bigl\{ (S_{pq} - \rho_{pq}) - (S_{pp} - 1 )(S_{pq} - 1 )- (S_{qq} - 1 )(S_{pq} - 1 )\Bigr\}\\
%&- \rho_{pq}^2 \Bigl\{ (S_{pp} - 1)+ (S_{qq} - 1) + ( S_{pp}- 1)^2 +  ( S_{pp}- 1)^2 + (S_{qq}- 1)(S_{qq}- 1)\Bigr\} +
%\end{align}

For each pair $p < q$, the main difference between $II_{pq, 1}$  and $II_{pq, 2}$ is that when $\lambda_3 \not = 2$,  all the coefficients $\frac{\partial^{\boldsymbol \lambda}f(1, 1, \rho_{pq})}{{\boldsymbol \lambda}!}$ appearing in the second term of  \eqref{II_pq2} can be bounded by either $|\rho_{pq}|$ or $\rho_{pq}^2$ up to some multiplicative constants. This makes proving the useful bound for $\mathbb{E}[II_2^2]$  in terms of the norm $\|\cormat - \iden\|_F$ amenable to the straightforward approach of squaring and taking expectation.
%a less laborious task than doing so for $\mathbb{E}[(II_1 - \frac{m(m-1)}{2n})^2]$.
Thus we shall defer the proof of \lemref{II_2bdd} to  \appref{ProofvarBdd} and address the bound in \lemref{II_1bdd} for the rest of this section.
%
%  It turns out because of this, with the help of a few technical lemmas in the appendix, a  useful bound on $II_2^2$ is comparatively easier to derive ... , and we hereby state it in the following lemma.

We will start with the fact that
\begin{equation} \label{bddbytwopieces}
\mathbb{E}\left[\left(II_1 - \frac{m(m-1)}{2n}\right)^2\right] \leq 2 \left\{\text{Var}[II_1] + \left(\mathbb{E}[II_1] -  \frac{m(m-1)}{2n}\right)^2\right\}
\end{equation}
and form estimates for the terms on the right hand side.
% Recognising that the triples ${\boldsymbol \lambda} = (\lambda_1, \lambda_2, \lambda_3) $ indexing the second sum of \eqref{II_pq1} are
%$(1, 0, 2)$, $(0, 1, 2)$, $(2, 0, 2)$ and $(0, 2, 2)$,
% by \lemref{df} we explicityly write $II_{1, pq}$ as
%\begin{align}
%II_{1, pq} =&\ \ n^{-2}\sum_{i = 1}^n (X_{pi}X_{qi} - \rho_{pq})^2
%-   (S_{pp} - 1)(S_{pq} - \rho_{pq})^2-   (S_{qq} - 1)(S_{pq} - \rho_{pq})^2 \notag\\
%&\hspace{3 cm} +    (S_{pp} - 1)^2(S_{pq} - \rho_{pq})^2 +  (S_{qq} - 1)^2(S_{pq} - \rho_{pq})^2 \label{term1}\notag\\
%= & \ \ n^{-2}\sum_{i = 1}^n (X_{pi}X_{qi} - \rho_{pq})^2 \\
%& - n^{-3}\sum_{i, j, k = 1}^n  (X^2_{pi} - 1)(X_{pj} X_{qj} - \rho_{pq})(X_{pk} X_{qk} - \rho_{pq})  \label{term2}\\
%&- n^{-3}\sum_{i, j, k = 1}^n  (X^2_{qi} - 1)(X_{pj} X_{qj} - \rho_{pq})(X_{pk} X_{qk} - \rho_{pq})  \\
%&+ n^{-4}\sum_{i, j, k, l = 1}^n  (X^2_{pi} - 1)(X^2_{pj} - 1)(X_{pk} X_{qk} - \rho_{pq})(X_{pl} X_{ql} - \rho_{pq}) \\
%&+ n^{-4}\sum_{i, j, k, l = 1}^n  (X^2_{qi} - 1)(X^2_{qj} - 1)(X_{pk} X_{qk} - \rho_{pq})(X_{pl} X_{ql} - \rho_{pq})  \label{term5}
%\end{align}
To understand the mean and variance of $II_1$, it is more instructive to first recognize that each term in \eqref{II_pq1} can be written as a U-statistic of degree $4$.
%To understand the  mean and variance of $II_1$, it is more instructive to first recognize that each of the terms from \eqref{term1} to \eqref{term5} can in fact be written as a U-statistic of degree $4$.
For instance, for any four distinct indices $1 \leq i, j, k, l \leq n$, if we only treat ${\bf X}_{pq, i} = (X_{pi}, X_{qi})', \dots, {\bf X}_{pq, l} = (X_{pl}, X_{ql})'$ as a four tuple in $\mathbb{R}^2$, the function
\begin{equation} \label{h1pq}
h_{1, pq}({\bf X}_{pq, i},{\bf X}_{pq, j}, {\bf X}_{pq, k}, {\bf X}_{pq, l}):= \frac{{n \choose 4}}{n^{2}{n-1 \choose 3} }\sum_{i' \in \{i, j, k, l\}}\left\{ (X_{p i'}X_{q i'} - \rho_{pq})^2 \right\},
\end{equation}
is symmetric in its four arguments, and
the first term in \eqref{II_pq1} can be written as the U-statistic
\begin{equation} \label{UstatRepresentation1}
n^{-2}\sum_{i = 1}^n (X_{pi}X_{qi} - \rho_{pq})^2   = \displaystyle{n \choose 4}^{-1} \sum h_{1, pq}({\bf X}_{pq, i}, {\bf X}_{pq, j}, {\bf X}_{pq, k}, {\bf X}_{pq, l})
\end{equation}
where the summation on the right hand side is over all distinct unordered qradruples $i, j, k , l$ that can be formed from $[n]$.
We note that the factor $n-1 \choose 3$ appears as a denominator in \eqref{h1pq} because for each $i \in \{1, \dots, n\}$, the summand $(X_{pi} X_{qi} - \rho_{pq})^2$ will appear only once on the left hand side of \eqref{UstatRepresentation1}, while by the definition of $h_{1, pq}$ it will appear in $n-1 \choose 3$ kernels that are summed over on the right hand side of  \eqref{UstatRepresentation1} (Since for each $i$, there will be $n-1 \choose 3$ choices of $j, k, l$ to form a quadruple $(i, j, k, l)$ from $\{1, \dots, n\}$). Thus,  the factor $n-1 \choose 3$ appears as a denominator in definition \eqref{h1pq} to account for the multiple counting.

Note that the other terms of the form $\frac{\partial^{\boldsymbol \lambda} f(1, 1, \rho_{pq})}{{\boldsymbol \lambda}!} \bar{S}_{pp}^{\lambda_1}\bar{S}_{qq}^{\lambda_2}\bar{S}_{pq}^{\lambda_3}$ in \eqref{II_pq1} are indexed by ${\boldsymbol \lambda}$ equal to $(1, 0, 2)$, $(0, 1, 2)$, $(2, 0, 2)$, $(0, 2, 2)$. These terms can be represented as U-statistics of degree $4$ using a similar strategy: With four \emph{distinct} indices $i, j, k, l$ from $[n]$,
by defining the symmetric kernel function
\begin{align} \label{h2pq}
&h_{2, pq}({\bf X}_{pq, i},{\bf X}_{pq, j}, {\bf X}_{pq, k}, {\bf X}_{pq, l})\\
&:= \underbrace{{n \choose 4} \frac{n^{-3}}{ {n -3 \choose 1}}}_{O(1)}\sum_{\substack{\{i', j', k'\} \\\subset \{i, j, k, l\}\\ i', j', k' \text{ distinct } \\\text{ and unordered}}}
\sum_{\pi \in \mathcal{S}_3}
 \Biggl\{(X_{p\pi(i')}^2 - 1)(X_{p\pi(j')}X_{q\pi(j')} - \rho_{pq})(X_{p\pi(k')}X_{q\pi(k')} - \rho_{pq}) \Biggl\}\notag\\
 &+ \underbrace{{n \choose 4}\frac{ n^{-3}}{{ n - 2\choose 2}}}_{O(n^{-1})}  \sum_{\substack{\{i', j'\} \\\subset \{i, j, k, l\}\\i', j' \text{ distinct } \\ \text{ and unordered}}} \sum_{\pi \in \mathcal{S}_2}\Biggl\{ (X_{p \pi( i')}^2 - 1)(  X_{p \pi(j')}X_{q\pi( j')} - \rho_{pq})^2\notag \\
&\hspace{4cm} + 2 (X_{p \pi( i')}^2 - 1) (X_{p \pi(i')}X_{q\pi( i')} - \rho_{pq})(X_{p \pi(j')}X_{q\pi(j' )} - \rho_{pq})\Biggr\}  \notag\\
&+ \underbrace{ {n \choose 4}\frac{n^{-3}}{ {n -1 \choose 3}} }_{O(n^{-2})} \sum_{i' \in \{i, j, k, l\}}  \Biggl\{ (X_{p i'}^2 - 1) (X_{p i'}X_{q i'} - \rho_{pq})^2\Biggl\}   \notag,
\end{align}
for ${\boldsymbol \lambda} = (1, 0, 2)$, where above we interpret $\pi$ as permutation functions on  distinct elements,  we have the  U-statistic  representation of degree $4$
\begin{align}
&\frac{\partial^{(1, 0, 2)}f(1, 1, \rho_{pq})}{(1, 0, 2)!}\bar{S}_{pp} \bar{S}_{pq}^2 \notag\\
 &= - n^{-3}\sum_{\tilde{i}, \tilde{j}, \tilde{k} = 1}^n  (X^2_{p\tilde{i}} - 1)(X_{p\tilde{j}} X_{q\tilde{j}} - \rho_{pq})(X_{p\tilde{k}} X_{q\tilde{k}} - \rho_{pq}) \label{line2}\\
&= - {n \choose 4}^{-1} \sum_{\substack{\text{unordered}\\  \& \text{ distinct } \\ i, j, k , l \\ \text{ from  } [n]}} h_{2, pq} ({\bf X}_{pq, i},{\bf X}_{pq, j}, {\bf X}_{pq, k}, {\bf X}_{pq, l}) \notag.
\end{align}
Note that \eqref{line2} simply comes from  \lemref{df}. What we have done here is that, for each term $  (X^2_{p\tilde{i}} - 1)(X_{p\tilde{j}} X_{q\tilde{j}} - \rho_{pq})(X_{p\tilde{k}} X_{q\tilde{k}} - \rho_{pq})$ in \eqref{line2}  with $\tilde{i}, \tilde{j}, \tilde{k}$ not necessarily distinct, we find any $4$ \emph{distinct} indices $i, j, k, l$ that contain $\tilde{i}, \tilde{j}, \tilde{k}$ as sets, and arrange the term into one of the three summands of order $O(1)$, $O(n^{-1})$ and $O(n^{-2)}$  in \eqref{h2pq} according to the actual set cardinality $|\{\tilde{i}, \tilde{j}, \tilde{k}\}|$, which can be equal to $1$, $2$ or $3$. Since there are ${ n - |\{\tilde{i}, \tilde{j}, \tilde{k}\}|\choose4-  |\{\tilde{i}, \tilde{j}, \tilde{k}\}|}$ choices of distinct $i, j, k, l$ that contain $\{\tilde{i}, \tilde{j}, \tilde{k}\}$ as sets, to account for the duplications we put the factors $n -3 \choose 1$, $n -2 \choose 2$, $n -1 \choose 3$ as denominators for the three summands in the definition \eqref{h2pq} of the kernel.
By a simple symmetry argument if we define the kernel
\begin{align}\label{h2barpq}
\bar{h}_{2, pq}({\bf X}_{pq, i}, {\bf X}_{pq, j}, {\bf X}_{pq, k}, {\bf X}_{pq, l}) := h_{2, pq}(\bar{\bf X}_{pq, i}, \bar{\bf X}_{pq, j}, \bar{\bf X}_{pq, k}, \bar{\bf X}_{pq, l})
\end{align}
where $\bar{\bf X}_{pq, i} := (X_{qi}, X_{pi})'$,
we have
\begin{align}
&\frac{\partial^{(0, 1, 2)}f(1, 1, \rho_{pq})}{(0, 1, 2)!}\bar{S}_{qq} \bar{S}_{pq}^2 \notag\\
&= - {n \choose 4}^{-1} \sum_{\substack{\text{unordered}\\  \& \text{ distinct } \\ i, j, k , l \\ \text{ from  } [n]}} \bar{h}_{2, pq} ({\bf X}_{pq, i},{\bf X}_{pq, j}, {\bf X}_{pq, k}, {\bf X}_{pq, l}) \notag.
\end{align}
In the same vein, for ${\boldsymbol \lambda}$ equals $(2, 0, 2)$ or $(0, 2, 2)$ and four \emph{distinct} indices $i, j, k, l$ from $[n]$, we leave it to the reader to check that one can define a symmetric kernel $h_{3, pq}$ of degree $4$ as shown in \appref{ProofvarBdd} such that
\[
\frac{\partial^{(2, 0, 2)}f(1, 1, \rho_{pq})}{(2, 0, 2)!}\bar{S}_{pp}^2   \bar{S}_{pq}^2
= {n \choose 4}^{-1}   h_{3, pq}({\bf X}_{ pq, i}, {\bf X}_{ pq, j}, {\bf X}_{ pq, k}, {\bf X}_{ pq, l})
\]
and
\[
\frac{\partial^{(0, 2, 2)}f(1, 1, \rho_{pq})}{(0, 2, 2)!}\bar{S}_{qq}^2 \bar{S}_{pq}^2
={n \choose 4}^{-1} \bar{h}_{3, pq}({\bf X}_{ pq, i}, {\bf X}_{ pq, j}, {\bf X}_{ pq, k}, {\bf X}_{ pq, l}),
\]
where
\begin{align}\label{h3barpq}
\bar{h}_{3, pq}({\bf X}_{pq, i}, {\bf X}_{pq, j}, {\bf X}_{pq, k}, {\bf X}_{pq, l}) := h_{3, pq}(\bar{\bf X}_{pq, i}, \bar{\bf X}_{pq, j}, \bar{\bf X}_{pq, k}, \bar{\bf X}_{pq, l}).
\end{align}
Letting ${\bf X}_i = (X_{1i}, \dots, X_{mi})'$ denote the entire $i$-th sample, we have  the degree-$4$ U-statistic representation for $II_{1}$:
\begin{equation} \label{UstatII2}
II_{1} = {n \choose 4}^{-1}  \sum_{ 1\leq i < j < k < l \leq n  }h({\bf X}_i, {\bf X}_j,{\bf X}_k, {\bf X}_l),
\end{equation}
where
\begin{multline}\label{kernelh}
h({\bf X}_i, {\bf X}_j,{\bf X}_k, {\bf X}_l) :=\\
 \sum_{1 \leq p < q \leq m} (h_{1, pq} - h_{2, pq} - \bar{h}_{2, pq} + h_{3, pq} + \bar{h}_{3, pq} )({\bf X}_{pq, i},{\bf X}_{pq, j}, {\bf X}_{pq, k}, {\bf X}_{pq, l}).
\end{multline}
Hence,
\begin{multline*}
\mathbb{E}[II_1] =\\
 \sum_{1 \leq p < q \leq m}  \mathbb{E}\left[(h_{1, pq} - h_{2, pq} - \bar{h}_{2, pq} + h_{3, pq} + \bar{h}_{3, pq} )({\bf X}_{pq, 1}, {\bf X}_{pq, 2}, {\bf X}_{pq, 3}, {\bf X}_{pq, 4})\right].
\end{multline*}
The expectation for each of $h_{1, pq}(\cdot), h_{2, pq}(\cdot) , h_{3, pq}(\cdot)$  in the preceding display can be computed by taking expectation for each of the product terms appearing in $\{\cdot\}$ in definitions  \eqref{h1pq}, \eqref{h2pq} as well as the counterparts in the definition of $h_{3, pq}$ in \appref{ProofvarBdd} (Note that quite a few of these expectations are simply zero due to independence of samples). Exploiting symmetry the same can be done  for  \eqref{h2barpq} and  \eqref{h3barpq}.  In principle, these higher-order normal moments can all be obtained by repeatedly applying Isserlis's theorem (\thmref{isserlis}) laboriously. With symbolic computational softwares such as \texttt{mathematica} they can however be much more effortlessly computed. These computations lead to
%
%\begin{lemma} [Expectation of $II_1$] \label{lem:expII1}
\begin{multline} \label{bddexpII1}
\mathbb{E}[II_1]
%&= \sum_{p < q }{n \choose 4} \left \{\frac{4}{{n-1 \choose 3}n^2}  (1 + \rho_{pq}^2)  - \frac{16}{n^3  {n-1 \choose 3}} (1 + 3 \rho_{pq}^2) +  \frac{48}{ n^4 {n-2 \choose 2}} (1+ 5\rho_{pq}^2)+  \frac{80}{{n-1 \choose 3}n^4}  (1 +  5 \rho_{pq}^2)\right\}\\
= \sum_{1 \leq p<q \leq m} \frac{16 + n^2 + (80 + 8n + n^2) \rho_{pq}^2}{n^3} \\
= \frac{m(m-1)}{2n} + O(n^{-1}) \|\cormat - \iden\|^2_F + O(m^2/n^3)
\end{multline}
%\end{lemma}
and further details are given in \appref{ProofvarBdd}.
As a direct consequence of \citet{Hoeffding1948}'s classical result on the variance of U-statistics, we also have the bound

%\begin{lemma}[Variance of a U-statistic, \citet{Hoeffding1948}] \label{lem:Ustat_var}
% For the symmetric kernel $h$ i.i.d data ${\bf X}_1, \dots, {\bf X}_n$, the U-statistic
%\[
%{n \choose 4}^{-1} \sum_{1 \leq i < j < k < l \leq n}h({\bf X}_i, {\bf X}_j, {\bf X}_k , {\bf X}_l)
%\]
%has variance given by
%\[
%{n \choose 4}^{-1} \sum_{c = 1}^4
% {4 \choose c} { n - 4 \choose 4 -c } \zeta_c ,\]
%where for $c = 1, \dots, 4$,
%\[
%\zeta_c := \mathbb{E}[g_c({\bf X}_1, \dots, {\bf X}_c)^2 ]
%\]
% and the functions $g_c : (\mathbb{R}^m)^c \longrightarrow \mathbb{R}$ are defined as
%\[
%g_c (x_1, \dots, x_c) := \mathbb{E}[ h({\bf X}_1, \dots, {\bf X}_4) |{\bf X}_1 = x_1,\dots, {\bf X}_c = x_c] - \mathbb{E}[h({\bf X}_1, \dots, {\bf X}_4)].
%\]
% In particular, since ${n \choose 4}^{-1} {4 \choose c}{n-4 \choose 4 -c}$ is of the order $O(n^{-c})$, we have the bound
%\begin{equation} \label{II2zetabdd}
%\text{Var}[II_2] \lesssim    \sum_{c = 1}^4 n^{-c} \zeta_c
%\end{equation}
\begin{equation} \label{II2zetabdd}
\text{Var}[II_1] \lesssim    \sum_{c = 1}^4 n^{-c} \zeta_c,
\end{equation}
where
\[
\zeta_c := \mathbb{E}[g_c({\bf X}_1, \dots, {\bf X}_c)^2 ]
\]
 and the functions $g_c : (\mathbb{R}^m)^c \longrightarrow \mathbb{R}$, $c = 1, \dots, 4$,  are defined as
\begin{equation} \label{gc}
g_c (x_1, \dots, x_c) := \mathbb{E}[ h({\bf X}_1, \dots, {\bf X}_4) |{\bf X}_1 = x_1,\dots, {\bf X}_c = x_c] - \mathbb{E}[h({\bf X}_1, \dots, {\bf X}_4)].
\end{equation}
Hence, forming  estimates of the quantities $\zeta_1, \dots, \zeta_4$ can lead to an estimate of  $\text{Var}[II_1]$.

\begin{lemma} [Bound for the $\zeta_c$'s] \label{lem:zetabounds}
\begin{align}
\zeta_1 &\lesssim \frac{\|\cormat - \iden\|^4_F + m^2 ( 1+ \|\cormat - \iden\|_F^2) }{n^2} + \frac{m^4}{n^4}\label{zeta1est} \\
\zeta_2 &\lesssim \frac{m^3 (1 + \|\cormat - \iden\|_F)}{n^2}   +\frac{m^4}{n^4} \label{zeta2est}\\
\zeta_3 &\lesssim \|\cormat - \iden\|^4_F + m^2 ( 1+ \|\cormat - \iden\|_F^2) + \frac{m^4}{n^2 }  \label{zeta3est}\\
\zeta_4 &\lesssim  m^3(\|\cormat - \iden\|_F +1) + \frac{m^4}{n^2} \label{zeta4est}
%\zeta_1 &\lesssim  n^{-2}  (m^2 + 37 M ) + n^{-2}( m M + M^2) + n^{-2} (m^2 M + M^2) + O(m^3n^{-3}) \\
%\zeta_2 &\lesssim  n^{-2} m^3 \\
%\zeta_3 &\lesssim  m M + M^2 + M + m^2 M + \sqrt{M} m^3 n^{-1}\\
%\zeta_4 &\lesssim  m^3 \sqrt{M}
\end{align}

\end{lemma}
Again, proving these estimates involves repeatedly applying \thmref{isserlis} with the help of \texttt{mathematica} and the details will be deferred to \appref{ProofvarBdd}. We note that these estimates are by no means sharp, but suffice for our purpose. Putting \lemref{zetabounds} and \eqref{II2zetabdd} together, it is a routine task to check that
\begin{equation*} \label{II2bdd}
 \text{Var}[II_1] \lesssim o\left(\frac{m^{2 (1 - \gamma)}}{n^2}\right) \sum_{k = 0}^4 \|\cormat - I\|_F^k
%\frac{\|\cormat - \iden\|^4_F }{n^3} + \frac{m^2}{n^3}\left(1 + \|\cormat - \iden\|_F^2\right)  + \frac{m^4}{n^5} +  \left(\frac{m^3}{n^4} + \frac{1}{n^2}\right)\|\cormat - \iden\|_F   .
\end{equation*}
for any $0 < \gamma < 1/2$. This, together with  \eqref{bddbytwopieces}  and \eqref{bddexpII1},  proved \lemref{II_1bdd}.

\section{Conclusion}

In this paper, we studied the exact power of the Rao's score statistic for testing independence, under the asymptotic regime where both the dimension $m$ and sample size $n$ grow to infinity when the ratio $m/n$ is bounded.  A consequence of our main result is  that  the Rao's score test is minimax rate optimal under this regime, with respect to a signal size $\|\cormat - \iden\|_F$ of order $\sqrt{m/n}$.

While previous related work \citep{ChenShao} on the null theory only requires the random variables to have finite moments, our power analysis relied on the normality assumption in different ways. Via applications of the Isserlis' theorem on normal moments (\thmref{isserlis}), all the higher moment quantities appeared in the calculations for the terms $I$ and $II$ in \secsref{BEsec} and \secssref{varbddsec}  can be controlled in terms of $\|{\bf R} - {\bf I}_m\|_F$, a second moment quantity in the original variables $X_1, \dots, X_m$ per se. It is thus conceivable that one can replace normality with appropriate higher moment conditions by carefully keeping track of these calculations. The tail bound for $III$ in \lemref{ProbBddIII} relies on a maximal inequality applicable to sub-exponential random variables, which is true for the centered sample covariances $\bar{S}_{pq}$ when they are formed with normal data (see \appref{PftailIII}). When normality cannot be assumed, we expect that one can use more general maximal inequalities such as \citet[Lemma 8]{CCK} along with their consequential moment conditions. A final caveat for pursuing the non-normal generality is that one should consider the more common definition of the sample covariance in \eqref{nice2} when constructing their Pearson correlations. Comparing \eqref{nice} with \eqref{nice2}, the insertion of sample means will likely complicate the calculations to follow under our current proof strategy.

\section*{Acknowledgments}

We thank the referees for their valuable comments and suggestions. Qi-Man Shao's research is partially  supported by the grant Hong Kong RGC GRF14302515.

\appendix

\section{Probability tail bound of $III$} \label{app:PftailIII}

 We will prove the tail bound for $III$ in \lemref{ProbBddIII}. For $1 \leq p ,   q \leq m$, by a standard trick \citep[p.221]{BickelLevina2008}, for any $t > 0$, one can show the sub-exponential inequality
\[
P \left(  |\bar{S}_{pq} |>  t\right)\leq 4 \exp
\left(
\frac{-t^2}{n^{-1} 2 (1  + \rho_{pq})  (2(1  + \rho_{pq} )+ t)}
\right)
\]
under our assumptions at the beginning of \secref{main}. Then by the maximal inequality in \citet[Lemma 2.2.10]{VDVW} and a union bound, we have  for any $0 <c < 1/2$,
%
% Notice that $P( \|S - R\|_\infty >t) \leq C \frac{\log m /n + \sqrt{\log m/ n}}{t}$, which implies, by \citet[]{VDVW},  that for any positive constant $c < 1/2$ (to be determined later) and any $\epsilon > 0$,
\begin{equation} \label{bddmax}
P( \max_{1 \leq p ,  q \leq m} |\bar{S}_{pq}|> n^{-c}) \lesssim  n^{c - 1} \log m  + n^{c - 1/2}\sqrt{\log m }.
\end{equation}
Note that by the definition of $III$,
\begin{equation} \label{bddI2}
|III| \leq   \max_{1 \leq p ,  q \leq m} |\bar{S}_{pq}|^5 \sum_{1 \leq p < q \leq m}\sum_{{\boldsymbol \lambda} : |{\boldsymbol \lambda}| = 5} \frac{{|\rho_{pq} + k_{pq}\bar{S}_{pq}  |}^{2 - \lambda_1 }}{{|1 + k_{pq}\bar{S}_{pp}|}^{1 + \lambda_2 } {| 1 +k_{pq}\bar{S}_{qq}  |}^{1 + \lambda_3 }}
\end{equation}
for ${\boldsymbol \lambda} = (\lambda_1, \lambda_2 , \lambda_3)$. If $ \max_{1 \leq p ,  q \leq m} |\bar{S}_{pq}| \leq n^{-c}$, for all $1 \leq p, q \leq m$ it must be true that
\begin{equation} \label{bddind}
|\rho_{pq} + k_{pq}\bar{S}_{pq}| \leq 1 + n^{-c}, |1 +  k_{pq} \bar{S}_{pp}| \geq 1- n^{-c}
\end{equation}
since $k_{pq} \in (0, 1)$
Combining \eqref{bddmax}, \eqref{bddI2}, \eqref{bddind}, with probability larger than $1  - C ( n^{c - 1} \log m  + n^{c - 1/2}\sqrt{\log m }) $
\[
|III| \leq C  n^{-5c} \frac{m(m-1)}{2}\frac{(1 + n^{-c})^2}{(1 - n^{-c})^7} \leq C \frac{m^2}{n^{5c}}
\]
for large $m, n$.

\section{Technical tools} \label{app:appen}

In this section we will lay out the technical tools required to finish the proofs in the paper.

\begin{lemma} \label{lem:df}
Let $f$ be as defined in \eqref{fun}.   For any ${\boldsymbol \lambda} = (\lambda_1, \lambda_2, \lambda_3) \in \mathbb{N}^3_{\geq 0}$

\[
 \partial^{\boldsymbol \lambda} f(u_1, u_2, u_3) =
  \begin{cases}
       (-1)^{\lambda_1 + \lambda_2} \lambda_1 ! \lambda_2 !  \frac{u_3^2}{u_1^{1 + \lambda_1}u_2^{1 + \lambda_2}}  & \text{if }  \lambda_3 = 0 \\
   2 (-1)^{\lambda_1 + \lambda_2} \lambda_1 ! \lambda_2 !   \frac{u_3^{2 - \lambda_3}}{u_1^{1 + \lambda_1}u_2^{1 + \lambda_2}}   & \text{if } \lambda_3  = 1,  2 \\
0      & \text{if }\lambda_3  > 2
  \end{cases}
\]
\end{lemma}

\begin{theorem} [\cite{isserlis}] \label{thm:isserlis}
For any natural number $k \geq 1$, let $(Z_1, \dots, Z_{2k}) $ be a mean zero normal vector with covariance matrix $\cormat = (\rho_{pq})_{1 \leq p , q \leq 2k}$. Then
\[
\mathbb{E}[Z_1 \dots Z_{2k}] =  \sum \rho_{p_1 p_2}\dots \rho_{p_{2k-1} p_{2k}},
\]
where the summation is over all possible $\frac{(2k)!}{2^k k!}$ partitions of the indices $1, \dots, 2k$ into $k$ pairs $(p_1, p_2), \dots, (p_{2k-1}, p_{2k})$.
\end{theorem}

\begin{corollary} \label{cor:momentfact}
For any four indices $1 \leq p_1, q_1, p_2, q_2 \leq m$,
 \[
\mathcal{M}_{\substack{(p_d, q_d)\\ d \in [2]}}  := \mathbb{E}[(X_{p_1} X_{q_1} - \rho_{p_1q_1})(X_{p_2} X_{q_2} - \rho_{p_2 q_2})] = \rho_{p_1q_2} \rho_{q_1p_2} + \rho_{p_1p_2} \rho_{q_1q_2}
\]
\end{corollary}
\begin{proof}
A simple corollary of \thmref{isserlis}.
\end{proof}

\begin{lemma} \label{lem:prodOrderLemma}
For a fixed natural number $k$, suppose $1 \leq p_d, q_d, \leq m$, $d = 1, \dots, 2k$ are any $2k$ pairs of  variable indices. Then
\[
\mathbb{E}\left[ \prod_{d = 1}^{2k}\bar{S}_{p_d q_d}\right] = O(n^{-k}),
\]
where the $O(\cdot)$ term is uniform for all choices of  $1 \leq p_d, q_d, \leq m$, $d = 1, \dots, 2k$.
\end{lemma}
\begin{proof}
On expansion,
\begin{equation*}
\mathbb{E}\left[ \prod_{d = 1}^{2k}\bar{S}_{p_d q_d}\right] =
n^{-2k}\left\{ \sum_{ \substack{i_d = 1, \dots, n\\ d = 1, \dots, 2k}} \mathbb{E}\left[\prod_{d = 1}^{2k}( X_{  p_di_d}X_{  q_di_d} - \rho_{p_d q_d})\right] \right\},
\end{equation*}
so we only need to show the term in $\{\cdot\}$ on the right hand side above is a uniform $O(n^k)$ term. We note that, by independence, an expectation on the right hand of the preceding display can only be non-zero if
\begin{equation} \label{notexists}
\not\exists\quad d' \in [2k]  \text{  such that } i_{d'} \not = i_d \quad  \forall \quad  d  \in [2k]\backslash \{d'\}.
\end{equation}
One way that \eqref{notexists} may happen is when there is a permutation $\pi \in \mathcal{S}_{2k}$ such that
\begin{equation} \label{exists}
i_{\pi_d} = i_{\pi_{k+d}} , \text{ for all } d = 1, \dots, k .
\end{equation}
There can at most be $O(n^k)$ many combinations of $i_1, \dots, i_{2k}$ satisfying \eqref{exists} since when \eqref{exists} is true,  the set $\cup_{d = 1}^{2k}\{i_d\}$  can at most have $k$ elements leaving us with $O({n \choose k}) = O(n^k)$ many choices for the combination of $i_1, \dots, i_{2k}$. We note that when a configuration in \eqref{notexists} is such that the set $\cup_{d = 1}^{2k}\{i_d\}$ has cardinality exactly equal to $k$,
\begin{multline} \label{manypairings}
\mathbb{E}\left[\prod_{d = 1}^{2k}( X_{  p_di_d}X_{  q_di_d} - \rho_{p_d q_d})\right] \\
= \prod_{d = 1}^k \mathbb{E}[ ( X_{p_{\pi_d}} X_{q_{\pi_{d}}}- \rho_{p_{\pi_d}q_{\pi_{d}}}) ( X_{p_{\pi_{d+k}}} X_{q_{\pi_{d+k}}}- \rho_{p_{\pi_{d+k}}q_{\pi_{{d+k}}}})] .\\
= \prod_{d = 1}^k ( \rho_{p_{\pi_d}q_{\pi_{d + k}}} \rho_{q_{\pi_d}p_{\pi_{d + k}}} + \rho_{p_{\pi_d}p_{\pi_{d + k}}} \rho_{q_{\pi_d}q_{\pi_{d + k}}}),
\end{multline}
by Corollary~\ref{cor:momentfact}.
 One can also easily see that there are at most   $O(n^{k-1})$ many combinations of $i_1, \dots, i_{2k}$ other than ones satisfying \eqref{exists} that can lead to \eqref{notexists}. Hence by \thmref{isserlis} and our assumption at the beginning of \secref{main}, we have
\[
 \sum_{ \substack{i_d = 1, \dots, n\\ d = 1, \dots, 2k}} \mathbb{E}\left[\prod_{d = 1}^{2k}( X_{  p_di_d}X_{  q_di_d} - \rho_{p_d q_d})\right] = O(n^k),
\]
where the $O(\cdot)$ is uniform for all choices of $1 \leq p_d, q_d \leq m$.
\end{proof}

The next two lemmas on sums of products of the entries  in the population correlation matrix $\cormat  = (\rho_{pq})_{ 1 \leq p, q \leq m}$ are keys for finishing our proofs.
\begin{lemma} \label{lem:threebounds}
Suppose $\pi = (\pi_1, \dots, \pi_4)$ is a particular permutation of the four indices $p, q, r, s$, say, $\pi  = (p, r, s, q)$. The following estimates are true:

\begin{align}
 \sum_{\substack{1 \leq  p, q, r, s  \leq m\\ p, q,r, s\text{ all distinct}}} |\rho_{pq} \rho_{rs} \rho_{\pi_1\pi_2} \rho_{\pi_3\pi_4}|  &\lesssim \|\cormat - \iden\|_F^4 \label{5.1}\\
%\sum_{ 1 \leq p \leq m} \sum_{\substack{1 \leq q \leq m \\q \not = p}}  \sum_{\substack{1 \leq r \leq m\\r \not = p, q}}
\sum_{\substack{1 \leq p, q, r \leq m \\p, q, r \text{ all distinct}}}
 |\rho_{pq} \rho_{qr} \rho_{pr}| &\lesssim  \|\cormat - \iden\|_F^4 +  \|\cormat - \iden\|_F^2 \label{5.2}
%%\sum_{\substack{1 \leq r \leq m\\r \not = p, q}}\sum_{\substack{1 \leq q \leq m \\q \not = p}}
%\sum_{\substack{1 \leq p, q, r \leq m \\p, q, r \text{all distinct}}}
%| \rho_{pq}^2 \rho_{pr}^2| & \lesssim \|\cormat - \iden\|_F^4 \label{5.3}
\end{align}

\end{lemma}

\begin{proof}[Proof of \lemref{threebounds}]

With a slight abuse of notations, the expression ``$r \not= p , q$" means that $r$ is a number that is not equal to $p$ nor $q$.

By the fact that $2 |ab| \leq a^2 + b^2$ for all $a, b \in \mathbb{R}$,
\begin{align*}
 \sum_{ \substack{1 \leq p, q, r, s \leq m \\ p, q, r, s\text{ all distinct}}} |\rho_{pq} \rho_{rs} \rho_{\pi_1\pi_2} \rho_{\pi_3\pi_4}|
 &\leq \sum_{ \substack{1 \leq p, q, r, s \leq m \\ p, q, r, s\text{ all distinct}}} \left\{(\rho_{pq} \rho_{rs})^2 + (\rho_{\pi_1\pi_2} \rho_{\pi_3\pi_4})^2 \right\}\\
&\leq  2 \sum_{1 \leq p \not = q \leq m } \rho_{pq}^2 \sum_{1 \leq r \not = s \leq m} \rho_{rs}^2 = 2 \|\cormat - \iden\|_F^4,
\end{align*}
which proves \eqref{5.1}. Similarly,
\begin{align*}
&\sum_{\substack{1 \leq p, q, r \leq m \\p, q, r \text{ distinct}}} | \rho_{pq} \rho_{qr} \rho_{pr}| = \sum_{1 \leq p \leq m}  \sum_{\substack{1\leq q \leq m \\ q \not = p} }(|\rho_{pq}| \sum_{\substack{ 1 \leq r \leq m \\ r \not = p, q}}|\rho_{qr} \rho_{pr}| )\\
&\leq\sum_{\substack{ 1 \leq p, q \leq m \\p \not = q}} \rho_{pq}^2 + \sum_{\substack{ 1 \leq p, q \leq m \\p \not = q}} (\sum_{\substack{ 1 \leq r \leq m \\ r \not = p, q}}|\rho_{qr} \rho_{pr}|)^2 \qquad \text{ by \ \ $2 |ab| \leq a^2 + b^2$ for all $a, b \in \mathbb{R}$}\\
&= \|\cormat - \iden\|_F^2 + \sum_{ \substack{1 \leq p, q \leq m \\ p \not = q}} (\sum_{\substack {1 \leq r \leq m \\ r  \not = p, q}}\sum_{\substack{ 1 \leq s \leq m \\ s \not = p, q}}|\rho_{qr} \rho_{pr}\rho_{qs} \rho_{ps}|)\\
%&\leq  \|\cormat - \iden\|_F^2 + \sum_{1 \leq p \not = q \leq m} \sum_{\substack{ 1 \leq r \leq m\\r \not = p, q}}\sum_{\substack{ 1 \leq s \leq m\\s \not = p, q}}(\rho_{qr}\rho_{ps})^2+  \sum_{1 \leq p \not = q \leq m} \sum_{\substack{ 1 \leq r \leq m\\r \not = p, q}}\sum_{\substack{ 1 \leq s \leq m\\s \not = p, q}}(\rho_{pr}\rho_{qs})^2 \\
%&\leq  \left[\|R - I\|_F^2 + \sum_{1 \leq p \not= q \leq m} \left(\sum_{\substack{1 \leq r \leq m \\ r \not = p, q}}\rho_{qr}^2\right)\left(\sum_{\substack{1 \leq s \leq m \\ s \not = p, q}}\rho_{ps}^2\right)+  \sum_{1 \leq p \not= q \leq m} \left(\sum_{\substack{1 \leq r \leq m \\ r \not = p, q}}\rho_{pr}^2\right)\left(\sum_{\substack{1 \leq s \leq m \\ s \not = p, q}}\rho_{qs}^2\right)\right]\\
&\leq  \|\cormat - \iden\|_F^2 + 2\|\cormat - \iden\|_F^4,
\end{align*}
where the last inequality comes from a similar proof as the one for \eqref{5.1}.
%Lastly,
%\begin{align*}
%\sum_{1 \leq p \leq m}  \sum_{\substack{1 \leq q\leq m \\  q \not = p}}  \sum_{\substack{ 1 \leq r \leq m \\ r \not = p, q}} \rho_{pq}^2 \rho_{pr}^2 &\leq \sum_{1\leq p \leq m}( \sum_{\substack{ 1\leq q \leq m \\ q  \not= p}} \rho_{pq}^2) ( \sum_{ \substack{1 \leq r \leq m \\ r \not= p}} \rho_{pr}^2) \\
%&\leq \sum_{1 \leq p \leq m}  \sum_{\substack{1 \leq q \leq m \\q \not= p}} \rho_{pq}^2 \|\cormat - \iden\|_F^2 =\|\cormat - \iden\|_F^4
%\end{align*}
%proves \eqref{5.3}.
\end{proof}

\begin{lemma} \label{lem:bddOnSums}
For $k = 5, \dots, 8$,
\begin{enumerate}
\item
\[
\sum_{\substack{1 \leq p_d, q_d\leq m\\d = 1, \dots, 4\\ |\cup_{d = 1}^4 \{p_d, q_d\}| = k}} \prod_{d  = 1}^4|\rho_{p_d q_d} | =  O(m^4) \|\cormat - \iden\|_F^{k - 4}.
\]

\item If $\pi = (\pi_1, \dots, \pi_8)$ and $\tau = (\tau_1, \dots, \tau_8)$ are two fixed permutations of the eight indices $p_1, q_1, \dots, p_4, q_4$. For instance,  $\pi$ can be equal to, say,  $(p_1, p_4, q_3, q_2, p_2, q_1, q_4,  p_3)$. Then

\[
\sum_{\substack{1 \leq p_d, q_d\leq m\\d = 1, \dots, 4\\ |\cup_{d = 1}^4 \{p_d, q_d\}| = k}}\prod_{d' = 1, 3, 5, 7}|\rho_{\pi_{d'} \pi_{d'+1}}\rho_{{\tau_{d'}} {\tau_{d'+1}}}|= O(m^{8 - k}) \|\cormat - \iden\|_F^{2(k - 4)}
\]
\end{enumerate}
\end{lemma}
\begin{proof} [Proof of \lemref{bddOnSums}]
We first note that for $(ii)$, By the inequality that $2|ab| \leq  a^2 + b^2$ for all $a, b \in \mathbb{R}$, we have
\[
\sum_{\substack{1 \leq p_d, q_d\leq m\\d = 1, \dots, 4\\ |\cup_{d = 1}^4 \{p_d, q_d\}| = k}}\prod_{d' = 1, 3, 5, 7}|\rho_{\pi_{d'} \pi_{d'+1}}\rho_{{\tau_{d'}} {\tau_{d'+1}}}| \leq 2^{-1}\sum_{\substack{1 \leq p_d, q_d\leq m\\d = 1, \dots, 4\\ |\cup_{d = 1}^4 \{p_d, q_d\}| = k}}\prod_{d = 1}^4 \rho_{p_d q_d}^2,
\]
hence to show $(ii)$ it suffices to show
\begin{equation} \label{anotherwayfor2}
\sum_{\substack{1 \leq p_d, q_d\leq m\\d = 1, \dots, 4\\ |\cup_{d = 1}^4 \{p_d, q_d\}| = k}}\prod_{d = 1}^4 \rho_{p_d q_d}^2 =   O(m^{8 - k}) \|\cormat - \iden\|_F^{2(k - 4)}
\end{equation}
Given  $k \in \{5, \dots , 8\}$, when $k$ of the indices $p_1, q_1, \dots, p_4, q_4$ are distinct, it must be the case that  there exist $k - 4$ pairs of $(p_d, q_d)$ such that all indices from these $k - 4$ pairs are distinct elements from $[m]$. Without lost of generality we can assume these $k - 4$ pairs to be $(p_1, q_1), \dots, (p_{k-4}, q_{k-4})$, which contains a total of $2k - 8$ distinct indices, and for proving $(i)$ and \eqref{anotherwayfor2} it suffices to  show, respectively,
\begin{equation} \label{firstk-4pairs}
\sum_{\substack{1 \leq p_1, q_1, \dots, p_4 , q_4\leq m\\ p_1, q_1 \dots, p_{k-4} , q_{k - 4}\text{  distinct} \\ |\cup_{d = 1}^4 \{p_d, q_d\}| = k }}\prod_{d  = 1}^4|\rho_{p_d q_d} | =  O(m^4) \|\cormat - \iden\|_F^{k - 4}
\end{equation}
and
\begin{equation} \label{firstk-4pairs2}
\sum_{\substack{1 \leq p_1, q_1, \dots, p_4 , q_4\leq m\\ p_1, q_1 \dots, p_{k-4} , q_{k - 4}\text{ distinct} \\ |\cup_{d = 1}^4 \{p_d, q_d\}| = k }}
\prod_{d = 1}^4 \rho_{p_d q_d}^2
=
O(m^{8 - k}) \|\cormat - \iden\|_F^{2(k - 4)}.
\end{equation}
As all the $\rho$'s are bounded in absolute value by $1$,  summing over the other $k - (2k - 8) = 8 - k$ indices different from $p_1, q_1, \dots, p_{k-4}, q_{k-4}$ results  in a $O(m^{8 - k})$ term which gives
\begin{multline}\label{theother8 - k}
\sum_{\substack{1 \leq p_1, q_1, \dots, p_4 , q_4\leq m\\ p_1, q_1 \dots, p_{k-4} , q_{k - 4}\text{  distinct} \\ |\cup_{d = 1}^4 \{p_d, q_d\}| = k }}\prod_{d  = 1}^4|\rho_{p_d q_d} |
 =
O(m^{8 - k}) \sum_{\substack{  1 \leq p_d, q_d \leq m\\ d = 1, \dots, k - 4 \\p_1, q_1, \dots, p_{k-4}, q_{k-4} \\\text{are distinct}}} \prod_{d = 1}^{k - 4} |\rho_{p_dq_d}|.
\end{multline}
%
% When $k$ of $p_1, \dots, $Since there are only finitely many ways in which $k$ of the indices $p_1, \dots, p_8$ can be distinct, by symmetry and the fact that all $\rho$'s are less than 1 in absolute value, without loss of generality it suffices to show
%\[
%\sum_{\substack{1 \leq p_1, \dots, p_8 \leq m\\ k \text{ of } p_1, \dots, p_8 \text{ being distinct} \\}}|\rho_{p_1 p_2} \rho_{p_3 p_4} \rho_{p_5 p_6} \rho_{p_7 p_8}| =  O(m^4) \|\cormat - \iden\|_F^{k - 4}.
%\]
Now  since
\[
\sum_{1 \leq p \not = q \leq m } |\rho_{pq}| \leq   O(m)\|\cormat - \iden\|_F
\]
by standard norm inequality, evaluating the sum on the right hand side of  \eqref{theother8 - k} we further obtain
\begin{equation*}
\sum_{\substack{1 \leq p_1, q_1, \dots, p_4 , q_4\leq m\\ p_1, q_1 \dots, p_{k-4} , q_{k - 4}\text{  distinct} \\ |\cup_{d = 1}^4 \{p_d, q_d\}| = k }}\prod_{d  = 1}^4|\rho_{p_d q_d} |
 =
O(m^{8 - k})O(m^{k - 4}) \|\cormat - \iden\|_F^{k - 4}
\end{equation*}
which is exactly \eqref{firstk-4pairs}. Similarly, summing over the other $ 8 - k$ indices different from $p_1, q_1, \dots, p_{k-4}, q_{k-4}$ on the left hand side of \eqref{firstk-4pairs2}  results  in a $O(m^{8 - k})$ term and hence
\[
\sum_{\substack{1 \leq p_1, q_1, \dots, p_4 , q_4\leq m\\ p_1, q_1 \dots, p_{k-4} , q_{k - 4}\text{ distinct} \\ |\cup_{d = 1}^4 \{p_d, q_d\}| = k }}
\prod_{d = 1}^4 \rho_{p_d q_d}^2
= O(m^{8-k})  \sum_{\substack{  1 \leq p_d, q_d \leq m\\ d = 1, \dots, k - 4 \\p_1, q_1, \dots, p_{k-4}, q_{k-4} \\\text{are distinct}}} \prod_{d = 1}^{k - 4} \rho_{p_dq_d}^2.
\] Since $\sum_{1 \leq p \not = q \leq m } \rho_{p q}^2 = \|\cormat - \iden\|_F^2$, we get \eqref{firstk-4pairs2} by continuing from the preceding display.
\end{proof}

\section{Proofs for \secref{BEsec}} \label{app:bddforBs}

\subsection{Proof of \eqref{S2est}-\eqref{S4est}}
First, we show the estimates  in \eqref{S2est}-\eqref{S4est}. Note that by Corollary~\ref{cor:momentfact},
\begin{align}
\mathbb{S}(k) &= \sum_{\substack{1 \leq p_d < q_d \leq m\\ d = 1, 2\\|\cup_{d  = 1}^2\{p_d, q_d\}| = k}}  (\rho_{p_1q_2} \rho_{q_1p_2} + \rho_{p_1p_2} \rho_{q_1q_2})^2\notag\\
&= \sum_{\substack{1 \leq p_d < q_d \leq m\\ d = 1, 2\\|\cup_{d  = 1}^2\{p_d, q_d\}| = k}}  (\rho_{p_1q_2}^2 \rho_{q_1p_2}^2 + \rho_{p_1p_2}^2 \rho_{q_1q_2}^2 + 2\rho_{p_1q_2} \rho_{p_2q_1} \rho_{p_1p_2} \rho_{q_1q_2})\label{analyze}
\end{align}
for $k = 2, 3, 4$. Also recall that $\rho_{pp} = 1$ for all $1 \leq p \leq m$. We will analyze the sum in \eqref{analyze} for different $k$.

\eqref{S2est}: When $k = 2$, with $p_d < q_d$ for $d = 1, 2$, it must be that $p_1 = p_2$ and $q_1 = q_2$, and hence from  \eqref{analyze}
\[
\mathbb{S}(2)=
  \sum_{1 \leq p< q \leq m}(1 + 2 \rho_{pq}^2  + \rho_{pq}^4) = \frac{m(m-1)}{2} + O(\|{\bf R}  - {\bf I}_m\|_F^2 ).
\]
since $\sum_{1 \leq p < q \leq m} \rho_{pq}^2 = 2^{-1}\|\cormat - \iden\|_F^2$ and $\rho_{pq}^4 \leq \rho_{pq}^2$.

\eqref{S3est}: When $k = 3$, one possible configuration of $\cup_{d  = 1}^2\{p_d, q_d\}$ as a set with cardinality $3$ is that
\begin{equation} \label{configfork3}
p_1 = p_2,  \ \ p_1 < q_1,  \ \  p_1 < q_2,   \ \ \text{ and }  \ \ q_1 \not= q_2.
\end{equation}
Taking a sum just over the terms in \eqref{analyze} whose indices $p_1, q_1, p_2, q_2$ satisfy the configuration \eqref{configfork3} we get
\begin{align*}
 &\sum_{\substack{1 \leq p_d < q_d \leq m\\ d = 1, 2\\\cup_{d  = 1}^2\{p_d, q_d\} \text{ as in \eqref{configfork3}}}} (\rho_{p_1q_2}^2 \rho_{q_1p_2}^2 + \rho_{p_1p_2}^2 \rho_{q_1q_2}^2 + 2\rho_{p_1q_2} \rho_{p_2q_1} \rho_{p_1p_2} \rho_{q_1q_2})\\
&= \sum_{\substack{1 \leq p_d < q_d \leq m\\ d = 1, 2\\\cup_{d  = 1}^2\{p_d, q_d\} \text{ as in \eqref{configfork3}}}} (\rho_{p_1q_2}^2 \rho_{p_1 q_1}^2 +  \rho_{q_1q_2}^2 + 2\rho_{p_1q_2} \rho_{p_1q_1} \rho_{q_1q_2}) \quad \text{by $\rho_{p_1 p_1} = 1$}\\
& \lesssim \sum_{ \substack{ 1\leq p_1, q_1, q_2 \leq m \\  p_1, q_1, q_2 \text{ distinct}}}(  \rho_{p_1q_1}^2 + |\rho_{p_1q_2} \rho_{p_1q_1} \rho_{q_1q_2}|)\\
& \lesssim m\|{\bf R} - {\bf I}_m\|_F^2 + \|{\bf R} - {\bf I}_m\|_F^4,
\end{align*}
where the second last inequality is true because we enlarged the set of indices $p_1, q_1, q_2$ we are summing over and used the fact that
\[
 \sum_{ \substack{ 1\leq p_1, q_1, q_2 \leq m \\  p_1, q_1, q_2 \text{ distinct}}} \rho_{p_1q_2}^2 \rho_{p_1 q_1}^2 \lesssim  \sum_{ \substack{ 1\leq p_1, q_1, q_2 \leq m \\  p_1, q_1, q_2 \text{ distinct}}}\rho_{p_1 q_1}^2
\]
since any $\rho^2_{pq}$ is less than $1$, and the last inequality follows from \eqref{5.2}  and that \[\sum_{ \substack{ 1\leq p_1, q_1, q_2 \leq m \\  p_1, q_1, q_2 \text{ distinct}}} \rho_{q_1 q_2}^2  \lesssim \sum_{1 \leq p_1 \leq m} \|\cormat - \iden\|_F^2
\lesssim m\|\cormat - \iden\|_F^2.
\]
The same estimates can be proved for other set configurations of $\cup_{d  = 1}^2\{p_d, q_d\}$ similar to the one in  \eqref{configfork3}. Since there are only finitely many  such configurations, we get the estimate in \eqref{S3est}.

%By enlarging the index set we are summing over and using symmetry among the indices we have
%\begin{align*}
%\mathbb{S}(3)
%&\notag = \sum_{\substack{1 \leq p < q \leq m\\ 1 \leq r < s \leq m \\|\{p, q\}\cup\{r, s\}| = 3}} \rho_{ps}^2 \rho_{qr}^2 + \rho_{pr}^2 \rho_{qs}^2 + 2\rho_{ps} \rho_{qr} \rho_{pr} \rho_{qs} \\
% &\lesssim  \sum_{\substack{ 1 \leq p, q , r \leq m \\ p, q, r \text{ distinct}}}
%\rho_{pq}^2 \rho_{pr}^2 + \rho_{pq}^2 + 2|\rho_{pq} \rho_{pr}\rho_{qr}|  \lesssim  m\|{\bf R} - {\bf I}_m\|_F^2 + \|{\bf R} - {\bf I}_m\|_F^4,
%\end{align*}
%where we have used \lemref{threebounds}$(iii)$ in the last inequality.

\eqref{S4est}: By considering different configurations for the set $\cup_{d  = 1}^2\{p_d, q_d\}$ with cardinality $4$, from \eqref{analyze} we have
% as in the proof for \eqref{S3est} again by enlarging the index set of summation and symmetry we have
\begin{align*}
\mathbb{S}(4) &\lesssim \sum_{\substack{1  \leq p_1 ,q_1 , p_2, q_2 \leq m\\  p_1 ,q_1 , p_2, q_2 \text{ distinct}}} \rho_{p_1q_1}^2 \rho_{p_2q_2}^2 +| \rho_{p_1q_2}\rho_{q_1p_2}\rho_{p_1p_2} \rho_{q_1q_2}| \\
& \lesssim \|{\bf R} - {\bf I}_m\|_F^4,
\end{align*}
where the last inequality used \eqref{5.1} and
\[
\sum_{\substack{1  \leq p_1 ,q_1 , p_2, q_2 \leq m\\  p_1 ,q_1 , p_2, q_2 \text{ distinct}}} \rho_{p_1q_1}^2 \rho_{p_2q_2}^2 \lesssim \|\cormat - \iden\|_F^2\sum_{ \substack{1 \leq p_1 , q_1 \leq m \\ p_1, q_1 \text{ distinct }}} \rho_{p_1q_1}^2 \lesssim \|\cormat - \iden\|_F^4.
\]
\subsection{Proof of \eqref{afterIsserlis}}
In fact, the strategy we used in proving \eqref{S3est} will also lead to a quick proof of  the estimates for $\mathbb{T}(k)$, $k = 5, \dots, 8$ in \eqref{afterIsserlis}. Be definition,
\[
\mathbb{T}(k)  = \sum_{\substack{1 \leq p_d < q_d \leq m\\d= 1, 2, 3, 4 \\ |\cup_{d = 1}^4 \{p_d, q_d\}| = k}} \mathbb{E}\left[ \prod_{d \in [4]}(X_{p_d}X_{q_d} - \rho_{p_dq_d})\right].
\]
By expanding the product $\prod_{d \in [4]}(X_{p_d}X_{q_d} - \rho_{p_dq_d})$ at the end of the above equation and taking expectation with  respect to \thmref{isserlis}, one can see that
\begin{equation} \label{Tkbdd}
\mathbb{T}(k) \lesssim \sum_{\substack{1 \leq p_d < q_d \leq m\\d= 1, 2, 3, 4 \\ |\cup_{d = 1}^4 \{p_d, q_d\}| = k}}  \sum_{\pi \in \mathcal{S}_8}
|\rho_{\pi_1 \pi_2}\rho_{\pi_3 \pi_4}\rho_{\pi_5 \pi_6}\rho_{\pi_7 \pi_8}|,
\end{equation}
where here we interpret $\pi = (\pi_1, \dots, \pi_8)$ as a permutation of the eight indices $p_1, q_1, \dots, p_4, q_4$. When the permutation $\pi = (p_1, q_1, p_2, q_2, p_3, q_3, p_4, q_4)$, we have
\begin{multline}
\sum_{\substack{1 \leq p_d < q_d \leq m\\d= 1, 2, 3, 4 \\ |\cup_{d = 1}^4 \{p_d, q_d\}| = k}} |\rho_{\pi_1 \pi_2}\rho_{\pi_3 \pi_4}\rho_{\pi_5 \pi_6}\rho_{\pi_7 \pi_8}|
%&= \sum_{\substack{1 \leq p_d < q_d \leq m\\d= 1, 2, 3, 4 \\ |\cup_{d = 1}^4 \{p_d, q_d\}| = k}} \prod_{d  =1}^4 |\rho_{p_d q_d}| \notag\\
\\
\leq \sum_{\substack{1 \leq p_d ,  q_d \leq m\\d= 1, 2, 3, 4 \\ \text{ and }\\ k \text{ of } p_1, q_1, \dots, p_4, q_4 \\\text{ are distinct}}} \prod_{d  =1}^4 |\rho_{p_d q_d}|  = O(m^4) \|\cormat - \iden\|_F^{k-4}\label{C3ineq}
\end{multline}
by \lemref{bddOnSums}$(i)$.
Although \eqref{C3ineq} is only proved for $\pi = (p_1, q_1, p_2, \dots, p_4, q_4)$, a same bound for all other permutations easily generalize, which gives our estimate in \eqref{afterIsserlis} in light of \eqref{Tkbdd}.

\subsection{Proof of \eqref{firstBrace}-\eqref{thirdBrace}}

\ \

\eqref{firstBrace}: We first write
\begin{equation} \label{rewritesumP1}
 \sum_{\substack{1\leq p_d < q_d \leq m \\d = 1, 2, 3, 4}} \mathbb{P}_1 =  \sum_{k = 5}^8 \left\{\sum_{\substack{1\leq p_d < q_d \leq m \\d = 1, 2, 3, 4 :\\ |\cup_{d = 1}^4\{p_d, q_d \}| = k}} \mathbb{P}_1 \right\}+ O(m^4)
\end{equation}
%since
%\[
%\sum_{\substack{1\leq p_d < q_d \leq m \\d = 1, 2, 3, 4 :\\ |\{p_d, q_d : d = 1, \dots, 4\}| \leq 4}} \mathbb{P}_1
%\]
%is
where the $O(m^4)$ term comes from a remaining sum of $O(m^4)$ many universally bounded terms when $|\cup_{d = 1}^4\{p_d, q_d \}|  \leq 4$. By the definition of $\mathbb{P}_1$ in \eqref{P1}, Corollary~\ref{cor:momentfact} and \lemref{bddOnSums}$(i)$, it can be seen that for each $k = 5, \dots, 8$,
\[
\sum_{\substack{1\leq p_d < q_d \leq m \\d = 1, 2, 3, 4 :\\ |\cup_{d = 1}^4\{p_d, q_d\}| = k}} \mathbb{P}_1 \lesssim  \sum_{\substack{1 \leq p_d ,  q_d \leq m\\d= 1, 2, 3, 4 \\  |\cup_{d = 1}^4\{p_d, q_d\}| = k}} \prod_{d  =1}^4 |\rho_{p_d q_d}| = O(m^4) \|\cormat - \iden\|_F^{k - 4},
\]
giving \eqref{firstBrace} in light of \eqref{rewritesumP1}.
% On multiplying out the expression in \eqref{P1}, $P_1$ is a sum of finitely many terms each having the form
%\[
%\prod_{d = 1}^4\rho_{\pi(p_d)\pi(q_d)}
%\]
%for some $\pi \in \mathcal{S}_8$. Hence, by the triangular inequality and \lemref{bddOnSums}$(i)$,  we have
%\begin{equation} \label{1stbddonP1}
%\left|\sum_{\substack{1\leq p_d < q_d \leq m \\d = 1, 2, 3, 4 :\\ |\{p_d, q_d : d = 1, \dots, 4\}| = k}} P_1  \right|\leq   O(m^4) \|\cormat - \iden\|_F^{k - 4}
%\end{equation}
%for each $k = 5, \dots, 8$.
%, where we obtain the second equality by applying \lemref{bddOnSums}$(i)$, observing that from $P_1$ is a finite sum of terms each having the form
%\[
%\prod_{d = 1}^4\rho_{\pi(p_d)\pi(q_d)}
%\]
%and the triangular inequality.

\eqref{secondBrace} and \eqref{thirdBrace}:  Similar to \eqref{rewritesumP1} for $u = 1, 2, 3$, we can write
\begin{equation}\label{P1P2+P1P3+P1P4}
\sum_{\substack{1\leq p_d < q_d \leq m \\d = 1, 2, 3, 4}} \mathbb{P}_1 \mathbb{P}_u =
\sum_{k= 4}^8\sum_{\substack{1\leq p_d < q_d \leq m \\d = 1, 2, 3, 4 \\ |\cup_{d = 1}^4 \{p_d, q_d\}| =  k}}  \mathbb{P}_1 \mathbb{P}_u  + O(m^3).
\end{equation}
By Corollary~\ref{cor:momentfact} we get that $\mathbb{P}_1 \mathbb{P}_u$ is a finite sum of terms  each having the form
\[
\prod_{d'= 1, 3, 5, 7}\rho_{\pi_{d'} \pi_{d'+1}} \prod_{d' = 1, 3, 5, 7}\rho_{\tau_{d'} \tau_{d'+1}}
\]
for $\pi = (\pi_1, \dots, \pi_8)$ and $\tau =  (\tau_1, \dots, \tau_8)$ that are certain permutations of the $8$ indices $p_1, q_1, \dots, p_4, q_4$. As such, by \lemref{bddOnSums}$(ii)$, for a given $k = 5, \dots, 8$,
\begin{align} \label{k5on}
\sum_{k = 5}^8\sum_{\substack{1\leq p_d < q_d \leq m \\d = 1, 2, 3, 4 \\ |\cup_{d = 1}^4\{p_d, q_d\} | =  k}} \mathbb{P}_1 \mathbb{P}_u &\lesssim\sum_{k = 5}^8 O(m^{8 - k}) \|\cormat - \iden\|_F^{2(k - 4)} \notag \\
&= \sum_{k = 0}^3 O(m^k)\|\cormat - \iden\|_F^{2(4 - k)}.
\end{align}
%Hence we get that
%\begin{multline} \label{k5on}
%\sum_{k= 5}^8\sum_{\substack{1\leq p_d < q_d \leq m \\d = 1, 2, 3, 4 \\ |\cup_{d = 1}^4\{p_d, q_d\}| =  k}} \mathbb{P}_1 \mathbb{P}_u =\\
% \sum_{k = 5}^8O(m^{8 - k}) \|\cormat - \iden\|_F^{2(k - 4)} = \sum_{k = 0}^3 O(m^k)\|\cormat - \iden\|_F^{2(4 - k)}.
%\end{multline}
Given \eqref{P1P2+P1P3+P1P4} and \eqref{k5on} it remains to show
\begin{equation} \label{m4thing}
\sum_{\substack{1\leq p_d < q_d \leq m \\d = 1, 2, 3, 4 \\ |\cup_{d  = 1}^4 \{p_d, q_d \}| = 4}}  \mathbb{P}_1^2 =  {m \choose 2} {m-2 \choose 2} + O(m^3) \|\cormat - \iden\|_F
\end{equation}
and, for $u = 2, 3$,
\begin{equation} \label{u2u3thing}
\sum_{\substack{1\leq p_d < q_d \leq m \\d = 1, 2, 3, 4 \\ |\cup_{d  = 1}^4 \{p_d, q_d \}| = 4}}  \mathbb{P}_1 \mathbb{P}_u = O(m^3)\|\cormat - \iden\|_F
\end{equation}
to prove \eqref{secondBrace} and \eqref{thirdBrace}. To that end we make the following claim:

 \ \

\emph{Claim}. Suppose $\pi = (\pi_1, \dots, \pi_8)$ and $\tau = (\tau_1, \dots, \tau_8)$ are two given permutations of eight indices $p_1, q_1, \dots, p_4, q_4 \in [m]$. Then
\begin{equation} \label{claimeqt}
\sum_{\substack{1 \leq p_d < q_d \leq m \\ d = 1 , \dots, 4 \\ |\cup_{d = 1}^4 \{p_d, q_d\}| = 4}}
\left(\prod_{d'  = 1, 3, 5, 7} \rho_{\pi_{d'} \pi_{d' +1} }\rho_{\tau_{d'} \tau_{d' +1} } \right) = O(m^3) \|\cormat - \iden\|_F
\end{equation}
unless, as elements in $[m]$,
\begin{equation} \label{pi'tau'equal}
\pi_{d'} =  \pi_{d'+1},  \ \  \tau_{d'} =  \tau_{d'+1}
\end{equation}
 for all  $d'= 1, 3, 5, 7$ when $1 \leq p_d < q_d \leq m$ for all $d = 1, \dots, 4$ and $|\cup_{d = 1}^4 \{p_d, q_d\}|= 4$.

\ \

The proof of this claim will be left till the end of this section. Using this, we will first show \eqref{u2u3thing} for $u=2$ while the proof for $u =3$ follows similarly and is thus omitted.

%we can show that
%\begin{equation} \label{P1P3}
% \sum_{\substack{1\leq p_d < q_d \leq m \\d = 1, 2, 3, 4 \\\cup_{d  = 1}^4 |\{p_d, q_d \}| = 4}} \mathbb{P}_1\mathbb{P}_3 = O(m^3) \|\cormat - \iden\|_F ,
%\end{equation}
%\begin{equation} \label{P1P4}
%  \sum_{\substack{1\leq p_d < q_d \leq m \\d = 1, 2, 3, 4 \\\cup_{d  = 1}^4 |\{p_d, q_d \}| = 4}} \mathbb{P}_1\mathbb{P}_4 = O(m^3) \|\cormat - \iden\|_F,
%\end{equation}
%
%and
%\begin{equation} \label{P1P1}
% \sum_{\substack{1\leq p_d < q_d \leq m \\d = 1, 2, 3, 4 \\\cup_{d  = 1}^4 |\{p_d, q_d \}| = 4}} \mathbb{P}_1^2= {m \choose 2}{m - 2 \choose 2} + O(m^3) \|\cormat - \iden\|_F
%\end{equation}
%which together give \eqref{m4thing}.
By Corollary~\ref{cor:momentfact}, on expansion we get that the $  \sum_{\substack{1\leq p_d < q_d \leq m \\d = 1, 2, 3, 4 \\\cup_{d  = 1}^4 |\{p_d, q_d \}| = 4}} \mathbb{P}_1\mathbb{P}_2$ is  a finite sum of terms each having the form as in the left hand side of \eqref{claimeqt} with $\pi$ and $\tau$ NOT satisfying the description in \eqref{pi'tau'equal} of the claim. For example, by Corollary~\ref{cor:momentfact}, on expansion
\[
\mathbb{P}_1 = (\rho_{p_1q_2} \rho_{q_1p_2} + \rho_{p_1p_2} \rho_{q_1q_2})(\rho_{p_3q_4} \rho_{q_3p_4} + \rho_{p_3p_4} \rho_{q_3q_4})  = \rho_{p_1p_2} \rho_{q_1q_2} \rho_{p_3p_4} \rho_{q_3q_4} + \cdots
\]
\[
\mathbb{P}_2 = (\rho_{p_1q_3} \rho_{q_1p_3} + \rho_{p_1p_3} \rho_{q_1q_3})(\rho_{p_2q_4} \rho_{q_2p_4} + \rho_{p_2p_4} \rho_{q_2q_4})  = \rho_{p_1p_3} \rho_{q_1q_3}\rho_{p_2p_4} \rho_{q_2q_4} + \cdots,
\]
which leads to
\begin{equation} \label{eexpand}
 \sum_{\substack{1\leq p_d < q_d \leq m \\d = 1, 2, 3, 4 \\ |\cup_{d  = 1}^4 \{p_d, q_d \}| = 4}} \mathbb{P}_1\mathbb{P}_2 =  \sum_{\substack{1\leq p_d < q_d \leq m \\d = 1, 2, 3, 4 \\|\cup_{d  = 1}^4 \{p_d, q_d \}| = 4}} \left(\prod_{d = 1, 3, 5, 7} \rho_{\pi'_d \pi'_{d+1}}\rho_{\tau'_d \tau'_{d+1}}  \right)+ \cdots,
\end{equation}
where
\begin{equation} \label{pi'}
\pi' = (\pi'_1, \dots, \pi'_8) :=  (p_1, p_2, q_1, q_2, p_3, p_4, q_3, q_4)
\end{equation}
and
 \begin{equation}\label{tau'}
\tau' = (\tau'_1, \dots, \tau'_8):= (p_1, p_3, q_1, q_3, p_2, p_4, q_2, q_4)
\end{equation}
 and similar terms are omitted in $\cdots$ above. Note that when $|\cup_{d = 1}^4 \{p_d, q_d\}| = 4$ and $1 \leq p_d< q_d \leq m$, there must be a pair among   $\{(\pi'_d, \pi'_{d+1}), (\tau'_d, \tau'_{d+1}): d  = 1, 3, 5, 7\}$ that contains two distinct elements in $[m]$ due to a mismatch of the permutations $\pi'$ and $\tau'$: For if not in consideration of $\pi'$ it must be the case that $p_1 = p_2$, $q_1 = q_2$, $p_3 = p_4$ and $q_3 = q_3$ with $p_1, q_1, p_3, p_4$ being four distinct elements in $[m]$, but this will imply $\tau'_1 = p_1 \not= p_3 = \tau'_2$, a contradiction.  By the claim above the first term on the right hand side of \eqref{eexpand} equals to $O(m^3) \|\cormat - \iden\|_F$, where as the finitely many omitted terms $\cdots$ in \eqref{eexpand} can also be similarly bounded and \eqref{u2u3thing}  is proved.

We now show \eqref{m4thing}, again with Corollary~\ref{cor:momentfact}, we expand
\begin{multline}
\sum_{\substack{1\leq p_d < q_d \leq m \\d = 1, 2, 3, 4 \\ |\cup_{d  = 1}^4 \{p_d, q_d \}| = 4}}\mathbb{P}_1^2= \sum_{\substack{1\leq p_d < q_d \leq m \\d = 1, 2, 3, 4 \\  |\cup_{d  = 1}^4 \{p_d, q_d \}| = 4}}  \Bigl\{(\rho_{p_1p_2} \rho_{q_1q_2} \rho_{p_3p_4}\rho_{q_3q_4})^2 + \\
 (\rho_{p_1p_2} \rho_{q_1q_2} \rho_{p_3q_4}\rho_{q_3p_4})^2 +(\rho_{p_1q_2} \rho_{q_1p_2} \rho_{p_3p_4}\rho_{q_3q_4})^2 +(\rho_{p_1q_2} \rho_{q_1p_2} \rho_{p_3q_4}\rho_{q_3p_4})^2
\Bigr\} + \cdots \label{withOmitted},
\end{multline}
where we leave it to the reader to check that the omitted terms in $\cdots$ of \eqref{withOmitted} is of order $O(m^3)\|\cormat - \iden\|_F$ due to mismatch of permutations as in \eqref{pi'} and \eqref{tau'}. In fact, summing over the three terms on the second line of \eqref{withOmitted} also contribute a term of order $O(m^3)\|\cormat - \iden\|_F$: For example, summing over the last term on the second line of \eqref{withOmitted} equals
\begin{equation} \label{sumNoMismatch}
\sum_{\substack{1\leq p_d < q_d \leq m \\d = 1, 2, 3, 4 \\ |\cup_{d  = 1}^4 \{p_d, q_d \}| = 4}}\left(\prod_{d' = 1, 3, 5, 7}\rho_{\tilde{\pi}_{d'} \tilde{\pi}_{{d'}+1}}^2 \right)
\end{equation}
with
\[
\tilde{\pi} = (\tilde{\pi}_1, \dots, \tilde{\pi}_8) = (p_1, q_2, q_1, p_2,  p_3, q_4,q_3, p_4).
\]
When $ 1 \leq p_d < q_d \leq m $ and $|\cup_{d = 1}^4 \{p_d, q_d\}| = 4$, we cannot have $\tilde{\pi}_{d'} = \tilde{\pi}_{d'+1}$ for all $d'  = 1, 3, 5, 7$ and hence by the previous claim \eqref{sumNoMismatch} is of order $O(m^3) \|\cormat- \iden\|_F$. Hence it remains to show that summing over the terms on the first line of \eqref{withOmitted} gives
\begin{equation} \label{sumOverFirstTermwithOmitted}
\sum_{\substack{1\leq p_d < q_d \leq m \\d = 1, 2, 3, 4 \\  |\cup_{d  = 1}^4 \{p_d, q_d \}| = 4}} (\rho_{p_1p_2} \rho_{q_1q_2} \rho_{p_3p_4} \rho_{q_3q_4})^2 = {m \choose 2} {m - 2\choose 2} + O(m^3) \|\cormat - \iden\|_F
\end{equation}
 When $|\cup_{d = 1}^4 \{p_d, q_d\}| = 4$ with $p_d < q_d$ for all $d = 1, \dots, 4$, as a set $\cup_{d = 1}^4 \{p_d, q_d\}$ can take the configuration
\begin{equation} \label{m4config}
p_1 = p_2 ,  \ \ q_1 = q_2,  \ \ p_3 = p_4, \ \ q_3 = q_4.
\end{equation}
When \eqref{m4config} is true, $ \rho_{p_1p_2}^2 \rho_{q_1q_2}^2 \rho_{p_3p_4}^2 \rho_{q_3q_4}^2  = 1$, and hence
\begin{align}
&\sum_{\substack{1\leq p_d < q_d \leq m \\d = 1, 2, 3, 4 \\\cup_{d  = 1}^4 |\{p_d, q_d \}| = 4}} \rho_{p_1p_2}^2 \rho_{q_1q_2}^2 \rho_{p_3p_4}^2 \rho_{q_3q_4}^2\notag\\ &= \sum_{\substack{1 \leq p_3 < q_3 \leq m\\  \{p_1, q_1\}\cap\{p_3, q_3\} = \emptyset}} \sum_{1 \leq p_1 < q_1 \leq m} 1  + \sum_{\substack{1\leq p_d < q_d \leq m \\d = 1, 2, 3, 4 \\|\cup_{d  = 1}^4 \{p_d, q_d \}| = 4 \\ \cup_{d  = 1}^4 \{p_d, q_d \} \text{ NOT as in } \eqref{m4config}}} \rho_{p_1p_2}^2 \rho_{q_1q_2}^2 \rho_{p_3p_4}^2 \rho_{q_3q_4}^2\notag\\
&= {m \choose 2} {m -2 \choose 2} +  \sum_{\substack{1\leq p_d < q_d \leq m \\d = 1, 2, 3, 4 \\|\cup_{d  = 1}^4 \{p_d, q_d \}| = 4 \\ \cup_{d  = 1}^4 \{p_d, q_d \} \text{ NOT as in } \eqref{m4config}}} \rho_{p_1p_2}^2 \rho_{q_1q_2}^2 \rho_{p_3p_4}^2 \rho_{q_3q_4}^2. \label{notandyes}
\end{align}
For any configurations of the set $\cup_{d = 1}^4 \{p_d, q_d\}$ other than \eqref{m4config}, one of \begin{inparaenum} \item $p_1 \not = p_2$,  \item $q_1 \not = q_2$,  \item $p_3 \not = p_4$ or  \item $q_3 \not = q_4$ \end{inparaenum} must be true. For example, one such configuration is
\begin{equation} \label{otherconfig}
p_1 < p_2 =  q_1 < q_2 =  p_3 < p_4 =  q_3 < q_4.
\end{equation}
For this particular configuration, $(i)$ $p_1 \not = p_2$ is true.
Then we leave it to the reader to verify that by the same line of reasoning as in the proof of the claim below, we can show
\[
 \sum_{\substack{1\leq p_d < q_d \leq m \\d = 1, 2, 3, 4 \\|\cup_{d  = 1}^4 \{p_d, q_d \}| = 4 \\ \cup_{d  = 1}^4 \{p_d, q_d \} \text{ as in } \eqref{otherconfig}}}  \rho_{p_1p_2}^2 \rho_{q_1q_2}^2 \rho_{p_3p_4}^2 \rho_{q_3q_4}^2 = O(m^3) \|\cormat - \iden\|_F,
\]
where similar bounds can in fact be proved for all configurations of $\cup_{d = 1}^4 \{p_d , q_d\}$ other than \eqref{m4config}. This, togethers with \eqref{notandyes}, leads to \eqref{sumOverFirstTermwithOmitted}.

%
%\begin{multline} \label{otherconfigbound}
% \sum_{\substack{1\leq p_d < q_d \leq m \\d = 1, 2, 3, 4 \\|\cup_{d  = 1}^4 \{p_d, q_d \}| = 4 \\ \cup_{d  = 1}^4 \{p_d, q_d \} \text{ as in } \eqref{otherconfig}}}  \rho_{p_1p_2}^2 \rho_{q_1q_2}^2 \rho_{p_3p_4}^2 \rho_{q_3q_4}^2 \\
%\leq  \sum_{\substack{1\leq p_d < q_d \leq m \\d = 1, 2, 3, 4 \\|\cup_{d  = 1}^4 \{p_d, q_d \}| = 4 \\ \cup_{d  = 1}^4 \{p_d, q_d \} \text{ as in } \eqref{otherconfig}}}  |\rho_{p_1p_2}| \\ \leq O(m^2) \sum_{1 \leq p_1 < p_2 \leq m} |\rho_{pq}| \leq O(m^3 ) \|\cormat - \iden\|_F,
%\end{multline}
%where the first inequality is true since all $\rho$'s are bounded by $1$ in absolute value, the second inequality comes from summing the other $2$ indices different from $p_1, p_2$ which splits out a $O(m^2)$ term, and the third inequality comes form usual norm inequality. We also emphasize that the last inequality crucially relies on that configuration \eqref{otherconfig} satisfies $p_1 \not = p_2$. It can then be seen that for any other configuration different from \eqref{m4config}, a bound like \eqref{otherconfigbound} is true, and since there are only finitely many such configurations, from \eqref{notandyes} we get the bound
%\begin{equation} \label{bddfirsttermexpand}
%\sum_{\substack{1\leq p_d < q_d \leq m \\d = 1, 2, 3, 4 \\\cup_{d  = 1}^4 |\{p_d, q_d \}| = 4}} \rho_{p_1p_2}^2 \rho_{q_1q_2}^2 \rho_{p_3p_4}^2 \rho_{q_3q_4}^2 \leq  {m \choose 2} {m -2 \choose 2} +  O(m^3 ) \|\cormat - \iden\|_F.
%\end{equation}
%for the first term in \eqref{withOmitted}.

\begin{proof} [Proof of the claim]  Suppose \eqref{pi'tau'equal} is not true for some $d' \in \{1, 3, 5, 7\}$, and without loss of generality we assume $\pi_1 \not =   \pi_2$. Since $|\rho_{pq}| \leq 1$ for all $1 \leq p, q \leq m$, we have
\begin{align*}
&\sum_{\substack{1 \leq p_d < q_d \leq m \\  d = 1 , \dots, 4\\ |\cup_{d = 1}^4 \{p_d, q_d\}| = 4}}
\prod_{d'  = 1, 3, 5, 7}| \rho_{\pi_{d'} \pi_{d' +1} }\rho_{\tau_{d'} \tau_{d' +1} }| \leq \sum_{\substack{1 \leq \pi_{d} ,   \tau_{d'}  \leq m \\  d = 1 , \dots, 8\\ \pi_1 \not = \pi_2 \\ |\cup_{d = 1}^8 (\{\pi_{d}\} \cup  \{\tau_{d}\} ) | = 4}}\prod_{d'  = 1, 3, 5, 7}| \rho_{\pi_{d'} \pi_{d' +1} }\rho_{\tau_{d'} \tau_{d' +1} }|  \\
 &\leq \sum_{\substack{1 \leq \pi_{d} ,   \tau_{d}  \leq m ,  d = 1 , \dots, 8\\ \pi_1 \not = \pi_2 \\ |\cup_{d = 1}^8 (\{\pi_{d}\} \cup  \{\tau_{d}\} ) | = 4}}| \rho_{\pi_1 \pi_2}| = O(m^2 ) \sum_{ \substack{1 \leq  \pi_1 \not= \pi_2\leq m }} | \rho_{\pi_1 \pi_2}|= O(m^3) \|\cormat - \iden\|_F, \\
\end{align*}
aa desired.
\end{proof}

\section{Proofs for \secref{varbddsec}}  \label{app:ProofvarBdd}

Before finishing the proofs for \secref{varbddsec}, we first give the definition of the kernel $h_{3, pq}$ as mentioned in the main text:

\begin{align*}
&h_{3, pq}({\bf X}_{ pq, i}, {\bf X}_{ pq, j}, {\bf X}_{ pq, k}, {\bf X}_{ pq, l})\\
&:=  \underbrace{n^{-4} {n \choose 4}}_{O(1)} \sum_{\pi \in \mathcal{S}_4}  \Biggl\{ (X_{p\pi(i)}^2 - 1)(X_{p\pi(j)}^2 - 1)(X_{p\pi(k)} X_{q\pi(k)} - \rho_{pq} )(X_{p\pi(l)} X_{q\pi(l)} - \rho_{pq}) \Biggl\} \notag\\
&+ \underbrace{\frac{n^{-4}}{ {n  - 3\choose 1}}{n \choose 4}}_{O(n^{-1})}\sum_{\substack{ \{i', j', k'\}\\\subset \{i, j, k, l\} \\ i', j', k' \text{distinct}}} \sum_{\pi \in \mathcal{S}_3} \Biggl\{ (X_{p \pi (i')}X_{q \pi (i')} - \rho_{pq})^2(X_{p \pi (j')}^2 - 1)(X_{p \pi (k')}^2 - 1)  \notag\\
&\hspace{2cm} +
(X_{p \pi (i')}^2 - 1)^2(X_{p\pi ( j')}X_{q \pi (j')} - \rho_{pq})(X_{p \pi (k')}X_{q \pi (k')} - \rho_{pq})
  \notag\\
& \hspace{2 cm}+ 4 (X^2_{p \pi(i')} - 1)( X_{p \pi(i')} X_{q \pi(i')} - \rho_{pq})(X^2_{p\pi(j')} - 1)( X_{p \pi(k')} X_{q \pi(k')} - \rho_{pq})\Biggl\}\notag  \\
& + \underbrace{\frac{n^{-4}}{{n-2\choose 2}}{n \choose 4}}_{O(n^{-2})} \sum_{\substack{\{i',j'\}\\ \subset \{i, j, k , l\} \\ i', j' \text{distinct}}}\sum_{\pi \in \mathcal{S}_2}  \Biggl\{(X_{p \pi(i')}^2 -  1)^2( X_{p \pi(j')}X_{q \pi(j')} - \rho_{pq})^2 \\
&\hspace{2cm}+
 2(X_{p \pi(i')}^2 - 1)(X_{p\pi( j')}^2 - 1)(X_{p\pi( j')}X_{q \pi(j')} - \rho_{pq})^2\notag \\
& \hspace{2cm}+ 2(X_{p \pi(i')}X_{q\pi( i')} - \rho_{pq})(X_{p \pi(j')}X_{q \pi(j')} - \rho_{pq})( X_{p\pi( j')}^2 - 1)^2  \notag \\
&\hspace{2cm}+ 2 (X_{p \pi(i')}^2 - 1)(X_{p\pi( i')}X_{q\pi( i')} - \rho_{pq})  (X_{p\pi( j')}^2 - 1)(X_{p \pi(j')}X_{q \pi(j')} - \rho_{pq})  \Biggr\}\notag \\
&+  \underbrace{\frac{n^{-4}}{{n-1\choose 3}}{n \choose 4}}_{O(n^{-3})} \sum_{i' \in \{i, j, k , l\}}\left\{ (X_{pi'}^2 - 1)^2 (X_{pi'}X_{qi'} - \rho_{pq})^2 \right\}\notag
\end{align*}

We now proceed with the remaining proofs.

\begin{proof}[Proof for \lemref{II_2bdd}]
Note that by definition,
\begin{multline} \label{II_2_rewrite}
II_2 = \sum_{1 \leq p < q \leq m} \bar{S}_{pp} \bar{S}_{qq} \bar{S}_{pq}^2+ \sum_{\substack{{\boldsymbol \lambda} \in \mathbb{N}^3_{\geq 0}:  \\1 \leq |{\boldsymbol \lambda}| \leq 4  \\ \lambda_3 \not = 2 }} \left \{\sum_{1 \leq p < q \leq m}\frac{\partial^{\boldsymbol \lambda} f(1, 1, \rho_{pq})}{{\boldsymbol \lambda}!} \bar{S}_{pp}^{\lambda_1} \bar{S}_{qq}^{\lambda_2} \bar{S}_{pq}^{\lambda_3} \right\}.
\end{multline}
Since there are only finitely many ${\boldsymbol \lambda}$ we are summing over for the second term in  \eqref{II_2_rewrite}, by the general fact that
$
2 |ab| \leq a^2 + b^2
$
, it suffices to show that, for $ {\boldsymbol \lambda } = (\lambda_1, \lambda_2,\lambda_3)$ with $1 \leq |{\boldsymbol \lambda}| \leq 4 $ and  $\lambda_3  \not=   2$, the quantities
\begin{equation} \label{termforeachlambda}
\mathbb{E}\left[\left(\sum_{1 \leq p < q \leq m}\frac{\partial^{\boldsymbol \lambda} f(1, 1, \rho_{pq})}{{\boldsymbol \lambda}!} \bar{S}_{pp}^{\lambda_1} \bar{S}_{qq}^{\lambda_2} \bar{S}_{pq}^{\lambda_3} \right)^2\right] ,
%\\ O(n^{-|{\boldsymbol \lambda}|})\left(\sum_{\substack{1 \leq p < q \leq m \\1 \leq r < s \leq m}}\frac{\partial^{\boldsymbol \lambda} f(1, 1, \rho_{pq})\partial^{\boldsymbol \lambda} f(1, 1, \rho_{rs})}{ ({\boldsymbol \lambda}!)^2} \right),
\end{equation}
as well as
\begin{equation} \label{termfor112}
\mathbb{E}\left[ \left( \sum_{1 \leq p < q \leq m} \bar{S}_{pp} \bar{S}_{qq} \bar{S}_{pq}^2\right)^2 \right],
\end{equation}
can be bounded by the right hand side of \eqref{bddinII_2bdd} up to some multiplicative constants.  We will first show it for \eqref{termforeachlambda} case by case according to the multi-index degree of  $\boldsymbol \lambda$. The arguments rely on the fact that,  by \lemref{df},  it must be true that
\begin{equation} \label{bddbyrhoorrho2}
\left| \partial^{\boldsymbol \lambda} f(1, 1, \rho_{pq})  \right|\leq C | \rho_{pq}| \quad \text{ when } \quad \lambda_3 = 1
\end{equation}
and
\begin{equation}\label{bddbyrhoorrho2'}
 \left| \partial^{\boldsymbol \lambda} f(1, 1, \rho_{pq})  \right|\leq C  \rho_{pq}^2 \quad \text{ when } \quad \lambda_3 = 0
\end{equation}
for some constant $C > 0$. Consider $3$ cases:

$|{\boldsymbol \lambda}| = 3$ or $4$:
%On multiplying out the product in $\mathbb{E}[\cdot]$ and applying blah,  it must be true that   \eqref{termforeachlambda} can be bounded by
%\[
% O(n^{-|{\boldsymbol \lambda}|})\sum_{\substack{1 \leq p < q \leq m\\ 1 \leq r < s \leq m}}{\partial^{\boldsymbol \lambda} f(1, 1, \rho_{pq})} {\partial^{\boldsymbol \lambda} f(1, 1, \rho_{rs})}
%\]
With the facts in  \eqref{bddbyrhoorrho2} and  \eqref{bddbyrhoorrho2'}, with \lemref{prodOrderLemma}, \eqref{termforeachlambda} is less than
\begin{equation}
O(n^{-|{\boldsymbol \lambda}|}) \sum_{\substack{1 \leq p < q \leq m\\ 1 \leq r < s \leq m}} |\rho_{pq}| |\rho_{rs}| \quad \text{ or } \quad O(n^{-|{\boldsymbol \lambda}|}) \sum_{\substack{1 \leq p < q \leq m\\ 1 \leq r < s \leq m}} |\rho_{pq}|^2 |\rho_{rs}|^2.
\end{equation}
Respectively, by properties of norms they can be estimated by
\[
O(n^{-|{\boldsymbol \lambda}|}m^2) \|\cormat - \iden\|_F^2 \quad \text{ and } \quad O(n^{-|{\boldsymbol \lambda}|}) \|\cormat - \iden\|_F^4,
\]
which are both less than the right hand side of  \eqref{bddinII_2bdd} up to constants since $|\boldsymbol \lambda| = 3$ or $4$.
%Since $\sum_{1 \leq p < q \leq m}|\rho_{pq}|^2 = 2^{-1}\|R- I\|_F^2$ and $\sum_{1 \leq p < q \leq m}|\rho_{pq}| \leq  O(m)\|R- I\|_F$, we can express these estimates as
%\[
%O(n^{-|{\boldsymbol \lambda}|})  \|R- I\|_F^4 \quad \text{ or } \quad  O(n^{-|{\boldsymbol \lambda}|} m^2)\|R- I\|_F^2,
%\]
%either of which can be bounded by the right hand side of \eqref{termforeachlambda} since $|{\boldsymbol \lambda}| = 3$ or $4$.

$|{\boldsymbol \lambda}| = 1$: The only ${\boldsymbol \lambda} \in \mathbb{N}^3_{\geq 0}$ with $|{\boldsymbol \lambda}| = 1$ and $\lambda_3 \not = 2$ are $(1, 0, 0)$, $(0, 1, 0)$, $(0, 0, 1)$. When $\lambda_3 = 0$, by \eqref{bddbyrhoorrho2'} and \lemref{prodOrderLemma} the second moment quantity in \eqref{termforeachlambda} is bounded by
\[
O(n^{-1}) \sum_{\substack{1 \leq p < q \leq m\\ 1 \leq r < s \leq m}} \rho^2_{pq}\rho^2_{rs} = O(n^{-1})\|\cormat - \iden\|_F^4,
\]
less than the right hand side of \eqref{bddinII_2bdd}. When ${\boldsymbol \lambda} = (0, 0, 1)$, by \lemref{df},   \eqref{termforeachlambda}  equals
\begin{align}
\mathbb{E}\left[\left( \sum_{1\leq p < q \leq m}2 \rho_{pq} \bar{S}_{pq} \right)^2\right] &=4 \sum_{\substack{1\leq p < q \leq m \\ 1\leq r < s \leq m}}
\rho_{pq}\rho_{rs} \mathbb{E}[\bar{S}_{pq}\bar{S}_{rs}]. \notag\\
&= 4n^{-1}\left(\sum_{\substack{1\leq p < q \leq m \\ 1\leq r < s \leq m}}
\rho_{pq}\rho_{rs} \rho_{ps} \rho_{qr} +\sum_{\substack{1\leq p < q \leq m \\ 1\leq r < s \leq m}}
\rho_{pq}\rho_{rs} \rho_{pr} \rho_{qs}\right) \notag\\
&= O(n^{-1})(\|\cormat - \iden\|_F^4 + \|\cormat - \iden\|_F^2)\label{3rd2ndexpand},
\end{align}
%By blah, \eqref{1st2ndexpand} and \eqref{2nd2ndexpand} are both bounded by $O(n^{-1})\|R - I\|_F^4$. Since
where the second equality comes from the fact that
\begin{align*}
\mathbb{E}[\bar{S}_{pq}\bar{S}_{rs}] &= n^{-1} \mathbb{E}[( X_pX_q- \rho_{pq})( X_rX_s- \rho_{rs})] = n^{-1}( \rho_{ps} \rho_{qr} + \rho_{pr} \rho_{qs})
\end{align*}
due to the i.i.d.'ness of samples and Corollary~\ref{cor:momentfact}. To show the last equality, by exploiting symmetry it is easy to see that
\begin{multline} \label{whatshouldbetheno}
\sum_{\substack{1\leq p < q \leq m \\ 1\leq r < s \leq m}}
\rho_{pq}\rho_{rs} \rho_{ps} \rho_{qr} +\sum_{\substack{1\leq p < q \leq m \\ 1\leq r < s \leq m}}
\rho_{pq}\rho_{rs} \rho_{pr} \rho_{qs} \\
 \lesssim \sum_{k = 2}^4 \sum_{\substack{1\leq p , q, r, s \leq m \\ |\{p\} \cup \{q\} \cup \{r\} \cup \{s\} | = k}} |\rho_{pq}\rho_{rs} \rho_{pr} \rho_{qs}|.
\end{multline}
Observe that
\begin{align*}
 \sum_{\substack{1\leq p , q, r, s \leq m \\ |\{p\} \cup \{q\} \cup \{r\} \cup \{s\} | = 2}} |\rho_{pq}\rho_{rs} \rho_{pr} \rho_{qs}| &\lesssim \sum_{1 \leq p < q \leq m} \rho_{pq}^2\\
 \sum_{\substack{1\leq p , q, r, s \leq m \\ |\{p\} \cup \{q\} \cup \{r\} \cup \{s\} | = 3}} |\rho_{pq}\rho_{rs} \rho_{pr} \rho_{qs}| &\lesssim \sum_{\substack{1 \leq p ,  q, r\leq m\\ p, q, r \text{ distinct}}}| \rho_{pq}\rho_{pr}\rho_{qr}|\\
 \sum_{\substack{1\leq p , q, r, s \leq m \\ |\{p\} \cup \{q\} \cup \{r\} \cup \{s\} | = 4}}  |\rho_{pq}\rho_{rs} \rho_{pr} \rho_{qs}| &= \sum_{\substack{1\leq p , q, r, s \leq m \\  p, q, r,  s \text{ distinct}}}  |\rho_{pq}\rho_{rs} \rho_{pr} \rho_{qs}|.
\end{align*}
In light of \lemref{threebounds}, applying these bounds to \eqref{whatshouldbetheno} implies \eqref{3rd2ndexpand}.

$|{\boldsymbol \lambda}| = 2$: The only ${\boldsymbol \lambda}$'s with $|{\boldsymbol \lambda}| = 2$ and $\lambda_3 \not = 2$ are $(1, 1, 0), (2, 0, 0), (0, 2, 0),  (1, 0, 1)$ and $(0, 1, 1)$. For the first three of these since $\lambda_3 = 0$,  by \eqref{bddbyrhoorrho2'} and \lemref{prodOrderLemma} the quantity in \eqref{termforeachlambda} equals
\begin{equation} \label{1stlambda=2}
O(n^{-2}) \sum_{\substack{1 \leq p < q \leq m\\ 1 \leq r < s \leq m}} \rho^2_{pq}\rho^2_{rs} = O(n^{-2})\|\cormat - \iden\|_F^4.
\end{equation}
 For ${\boldsymbol \lambda} =  (1, 0, 1)$, with \lemref{df} the quantity in  \eqref{termforeachlambda} equals
\begin{equation} \label{lambda101}
4 \sum_{\substack{  1\leq p < q \leq m\\ 1\leq r < s \leq m }} \rho_{pq}\rho_{rs} \mathbb{E}[\bar{S}_{pp} \bar{S}_{pq} \bar{S}_{rr} \bar{S}_{rs}].
\end{equation}
By simple argument as in the proof of \lemref{prodOrderLemma} and Corollary~\ref{cor:momentfact}, it is not hard to see that
\begin{align}
&\mathbb{E}[\bar{S}_{pp} \bar{S}_{pq} \bar{S}_{rr} \bar{S}_{rs}] \notag\\
&= \frac{n(n-1)}{n^4}\Biggl\{ \mathbb{E}[(X_p^2 -1)(X_r^2 -1)]\mathbb{E}[(X_p X_q - \rho_{pq})(X_r X_s - \rho_{rs})]+\notag\\
 &\quad \mathbb{E}[(X_p^2 -1)(X_p X_q - \rho_{pq})]\mathbb{E}[(X_r^2 -1)(X_r X_s - \rho_{rs})] +\notag\\
& \quad \mathbb{E}[(X_p^2 -1)(X_r X_s - \rho_{rs})]\mathbb{E}[(X_r^2 -1)(X_p X_q - \rho_{pq})]\Biggr\}  +
 O(n^{-3}) \notag\\
&= O(n^{-2})  (  2 \rho_{pr}^2 (\rho_{ps} \rho_{qr} + \rho_{pr} \rho_{qs}) + 4\rho_{pq}\rho_{rs} + 4 \rho_{pr} \rho_{ps}\rho_{rp} \rho_{rq} )+ O(n^{-3}). \label{tobesubinto}
\end{align}
Substituting \eqref{tobesubinto} into \eqref{lambda101} we get

\begin{align}
&4 \sum_{\substack{  1\leq p < q \leq m\\ 1\leq r < s \leq m }} \rho_{pq}\rho_{rs} \mathbb{E}[\bar{S}_{pp} \bar{S}_{pq} \bar{S}_{rr} \bar{S}_{rs}] \notag\\
&= 4 \sum_{\substack{  1\leq p < q \leq m\\ 1\leq r < s \leq m }} \rho_{pq}\rho_{rs}
\{ O(n^{-2})  (  2 \rho_{pr}^2 (\rho_{ps} \rho_{qr} + \rho_{pr} \rho_{qs}) + 4\rho_{pq}\rho_{rs} + 4 \rho_{pr} \rho_{ps}\rho_{rp} \rho_{rq} )+ O(n^{-3})\}\notag \\
& \leq  O(n^{-2})\sum_{k = 2}^4 \sum_{\substack{1\leq p , q, r, s \leq m \\ |\{p\} \cup \{q\} \cup \{r\} \cup \{s\} | = k}} |\rho_{pq}\rho_{rs} \rho_{pr} \rho_{qs}| +  O(n^{-3})\left(\sum_{\substack{  1\leq p < q \leq m\\ 1\leq r < s \leq m }}|\rho_{pq} \rho_{rs}|\right)\notag\\
& \leq  O(n^{-2})(\|\cormat - \iden\|_F^4 + \|\cormat - \iden\|_F^2 )+  O(m^2 n^{-3}) \|\cormat  - \iden\|_F^2   \label{2ndlambda=2},
\end{align}
where the last two inequalities make use of similar arguments that prove \eqref{3rd2ndexpand}. By a symmetry argument the same estimate holds for ${\boldsymbol \lambda} = (0, 1, 1)$. Both \eqref{1stlambda=2} and \eqref{2ndlambda=2} are less than the right hand side of \eqref{bddinII_2bdd}.

It remains to form an estimate for \eqref{termfor112}. Note that
\begin{equation*}
\mathbb{E}\left[ \left( \sum_{1 \leq p < q \leq m} \bar{S}_{pp} \bar{S}_{qq} \bar{S}_{pq}^2\right)^2 \right] = \sum_{ \substack{1 \leq p_d < q_d \leq m\\d = 1, 2\\ |\cup_{d = 1}^2 \{p_d, q_d\}| = 4}}
\mathbb{E}\left[ \prod_{d = 1}^2\bar{S}_{p_dp_d} \bar{S}_{q_dq_d} \bar{S}_{p_dq_d}^2\right]
+ O\left(\frac{m^3}{n^4}\right),
\end{equation*}
where the $O(n^{-4})$ term comes from an argument similar to the proof of \lemref{prodOrderLemma}, and the $O(m^3)$ term comes from that the $O(m^3)$ many choices for $p_1, q_1, p_2, q_2$ when $|\cup_{d = 1}^2 \{p_d, q_d\}| \leq 3$. Hence it now suffices to show the first term on the right hand side of the preceding display is less than the RHS of  \eqref{bddinII_2bdd}. The argument is simple but a little tedious so we just sketch it here: By a similar argument as in the proof of \lemref{prodOrderLemma} we must have
\begin{multline} \label{trouble3}
\sum_{ \substack{1 \leq p_d < q_d \leq m\\d = 1, 2\\ |\cup_{d = 1}^2 \{p_d, q_d\}| = 4}}
\mathbb{E}\left[ \prod_{d = 1}^2\bar{S}_{p_dp_d} \bar{S}_{q_dq_d} \bar{S}_{p_dq_d}^2\right] = O\left(\frac{m^4}{n^5}\right) +  \\
\sum_{ \substack{1 \leq p_d < q_d \leq m\\d = 1, 2\\ |\cup_{d = 1}^2 \{p_d, q_d\}| = 4}}O\Biggl(  \frac{\sum_{\substack{k \in  \{  p_2, q_1, q_2\}}}\mathbb{E}\left[  (X_{p_1}^2 - 1) (X_k^2 - 1) \right] + \sum_{d'  =1}^2\mathbb{E}\left[ (X_{p_1}^2 - 1)(X_{p_{d'}}X_{q_{d'}} - \rho_{p_{d'}  q_{d'}})\right]}{n^4}\Biggr),
\end{multline}
where the expectations on the right come from the fact that $(X_{p_1}^2 - 1)$ must pair with one of $(X_{p_2}^2 - 1)$, $(X_{q_1}^2 - 1)$, $(X_{q_2}^2 - 1)$, $(X_{p_1} X_{q_1} - \rho_{p_1 q_1})$ and $(X_{p_2} X_{q_2} - \rho_{p_2 q_2})$  as in \eqref{manypairings}.  By Corollary~\ref{cor:momentfact},  for $k$ equals $p_2$, $q_1$ or $q_2$, it must be that
\begin{equation} \label{trouble1}
\sum_{ \substack{1 \leq p_d < q_d \leq m\\d = 1, 2\\ |\cup_{d = 1}^2 \{p_d, q_d\}| = 4}}
\mathbb{E}\left[  (X_{p_1}^2 - 1) (X_k^2 - 1) \right] = \sum_{ \substack{1 \leq p_d < q_d \leq m\\d = 1, 2\\ |\cup_{d = 1}^2 \{p_d, q_d\}| = 4}} 2 \rho_{p_1 k}^2 = O(m^2 \|\cormat - \iden\|_F^2);
\end{equation}
for $d'$ equals $1$ or $2$, it must be that
\begin{multline}\label{trouble2}
\sum_{ \substack{1 \leq p_d < q_d \leq m\\d = 1, 2\\ |\cup_{d = 1}^2 \{p_d, q_d\}| = 4}}
\mathbb{E}\left[ (X_{p_1}^2 - 1)(X_{p_{d'}}X_{q_{d'}} - \rho_{p_{d'}  q_{d'}}) \right] \\
= \sum_{ \substack{1 \leq p_d < q_d \leq m\\d = 1, 2\\ |\cup_{d = 1}^2 \{p_d, q_d\}| = 4}}\rho_{p_1 q_{d'}} \rho_{p_1 p_{d'}} + \rho_{p_1 p_{d'}}\rho_{p_1 q_{d'}} = O(m^3) \|\cormat - \iden\|_F.
\end{multline}
Substituting \eqref{trouble1}  and \eqref{trouble2} into \eqref{trouble3} gives that
\[
\sum_{ \substack{1 \leq p_d < q_d \leq m\\d = 1, 2\\ |\cup_{d = 1}^2 \{p_d, q_d\}| = 4}}
\mathbb{E}\left[ \prod_{d = 1}^2\bar{S}_{p_dp_d} \bar{S}_{q_dq_d} \bar{S}_{p_dq_d}^2\right]  =
O\left(\frac{m^4}{n^5} + \frac{m^2 \|\cormat - \iden\|_F^2}{n^4}+ \frac{m^3 \|\cormat - \iden\|_F}{n^4}\right)
\]
which gives us an estimate less than the one required.
\end{proof}

\begin{proof}[Proof of \eqref{bddexpII1} and \lemref{zetabounds}]
As described by the main text, with the help of the \texttt{Expectation} function provided by \texttt{mathematica}, we easily find  that
\[
\mathbb{E}[h_{1, pq}] = \frac{4 {n \choose 4}}{n^2 {n -1 \choose 3}} (1 + \rho_{pq}^2) = \frac{1 + \rho_{pq}^2}{n},
\]
\[
\mathbb{E}[h_{2, pq}] = \mathbb{E}[\bar{h}_{2, pq}] = \frac{8{n \choose 4}}{n^3 {n -1 \choose 3}} (1 + 3\rho_{pq}^2)=  \frac{2(1 + 3 \rho_{pq}^2)}{n^2}\ \ \text{ and}
\]
\[
\mathbb{E}[h_{3, pq}] =  \mathbb{E}[\bar{h}_{3, pq}] = \left(\frac{24}{n^4 {n -2 \choose 2}}+\frac{40}{{n-1 \choose 3} n^4} \right) {n \choose 4}(1 + 5\rho_{pq}^2)
= \frac{2(4+ n) (1 + 5 \rho_{pq}^2)}{n^3}
\]
for each pair $1 \leq p < q \leq m$. Collecting these and summing over all $1 \leq p < q \leq m$ gives the expectation of the kernel $h$ in \eqref{bddexpII1}. We will now prove \lemref{zetabounds}, first dealing with \eqref{zeta4est}.
 Note that $g_4 (\cdot)$ simply equals the kernel function $h$, in particular for a set of distinct sample indices $i, j, k, l \in [n]$ we have
\begin{align*}
&g_4({\bf X}_i, {\bf X}_j {\bf X}_k, {\bf X}_l) = h({\bf X}_i, {\bf X}_j {\bf X}_k, {\bf X}_l)\\
& =  O(1) \times\\
&\sum_{1 \leq p < q \leq m}\\
& \Biggl \{ - \sum_{\substack{\{i', j', k'\} \\\subset \{i, j, k, l\}\\ i', j', k' \text{ distinct } \\\text{ and unordered}}}
\sum_{\pi \in \mathcal{S}_3}
 \Bigl[(X_{p\pi(i')}^2 - 1)(X_{p\pi(j')}X_{q\pi(j')} - \rho_{pq})(X_{p\pi(k')}X_{q\pi(k')} - \rho_{pq}) \Bigr] \\
& -  \sum_{\substack{\{i', j', k'\} \\\subset \{i, j, k, l\}\\ i', j', k' \text{ distinct } \\\text{ and unordered}}}
\sum_{\pi \in \mathcal{S}_3}
 \Bigl[(X_{q\pi(i')}^2 - 1)(X_{p\pi(j')}X_{q\pi(j')} - \rho_{pq})(X_{p\pi(k')}X_{q\pi(k')} - \rho_{pq}) \Bigr]\\
& + \sum_{\pi \in \mathcal{S}_4}  \Bigl[ (X_{p\pi(i)}^2 - 1)(X_{p\pi(j)}^2 - 1)(X_{p\pi(k)} X_{q\pi(k)} - \rho_{pq} )(X_{p\pi(l)} X_{q\pi(l)} - \rho_{pq}) \Bigr] \\
& + \sum_{\pi \in \mathcal{S}_4}  \Bigl[ (X_{q\pi(i)}^2 - 1)(X_{q\pi(j)}^2 - 1)(X_{p\pi(k)} X_{q\pi(k)} - \rho_{pq} )(X_{p\pi(l)} X_{q\pi(l)} - \rho_{pq}) \Bigr]  \Biggr\} \\
&+ \sum_{1 \leq p < q \leq m} t_4 ({\bf X}_{pq, i}, {\bf X}_{pq, j}, {\bf X}_{pq, k}, {\bf X}_{pq, l},  \rho_{pq})O(n^{-1})\\
&:= \tilde{g}_4 ({\bf X}_{ i}, {\bf X}_{j}, {\bf X}_{k}, {\bf X}_{l}) +  \sum_{1 \leq p < q \leq m} t_4 ({\bf X}_{pq, i}, {\bf X}_{pq, j}, {\bf X}_{pq, k}, {\bf X}_{pq, l},  \rho_{pq})O(n^{-1}),
\end{align*}
by collecting the $O(1)$ terms in the definition of $h_{2,  pq}, \bar{h}_{2,  pq}, h_{3,  pq}, \bar{h}_{3,  pq}$, where $ t_4 ( \cdot)$ is just a fixed polynomial in ${\bf X}_{pq, i}, {\bf X}_{pq, j}, {\bf X}_{pq, k}, {\bf X}_{pq, l},  \rho_{pq}$ whose form is irrelevant to us. Using the fact that $2|ab| \leq a^2 + b^2$ for all $a, b  \in \mathbb{R}$, we have
\begin{equation} \label{zeta4lesssim}
\zeta_4 \lesssim \mathbb{E}[ \tilde{g}_4 ({\bf X}_{ i}, {\bf X}_{j}, {\bf X}_{k}, {\bf X}_{l})^2]
 + O\left(\frac{m^4}{n^2}\right).
\end{equation}
A key observation  is that upon squaring and taking expectation, $\mathbb{E}[ \tilde{g}_4 ({\bf X}_{ i}, {\bf X}_{j}, {\bf X}_{k}, {\bf X}_{l})^2]$ must be a sum of finitely many terms, where for some sample indices $\tilde{i}, \tilde{j} \in \{i, j, k, l\}$, each of these terms  can be ``$\lesssim$" bounded  by  the form
\begin{equation} \label{mysteryform}
\sum_{\substack{1 \leq p_d < q_d \leq m\\ d = 1, 2}}\mathbb{E}[A(p_1 q_1, \tilde{i})B(p_2 q_2, \tilde{j})],
\end{equation}
where for any sample index $i \in [n]$ and variable indice $p, q  \in [m]$, $A(p, q, i)$ and $B(p, q, i)$ may equal to  one of
\[
X^2_{pi} - 1,  \ \ X^2_{qi} - 1,\ \  X_{pi}X_{qi} - \rho_{pq}.
\]
Now if $\tilde{i} \not= \tilde{j}$, the form in \eqref{mysteryform} equals zero. If $\tilde{i} = \tilde{j}$, the form in \eqref{mysteryform} can be bounded by
\[
 \sum_{\substack{1 \leq p_d < q_d \leq m\\ d = 1, 2 \\|\cup_{d = 1}^2 \{p_d, q_d\}| = 4}}\mathbb{E}[A(p_1 q_1, \tilde{i})B(p_2 q_2, \tilde{i})]+ O(m^3),
\]
and by applying Corollary~\ref{cor:momentfact}, we leave it for the reader to check that the leading term in the preceding display must be ``$\lesssim$" bounded by $m^3\|\cormat - \iden\|_F$. Summarizing this gives us the bound in \eqref{zeta4est}.

%
%Now we observe that if we expand all the products in $\{\cdot\}$ in the definition of $\tilde{g}_4$, $\tilde{g}_4$ must be a sum of finitely many terms each having the form
%\[
%\sum_{1 \leq p< q \leq m} \left(\text{a non-constant monomial in } \rho_{pq}, {\bf X}_{pq, i}, {\bf X}_{pq, j}, {\bf X}_{pq, k}, {\bf X}_{pq, l}\right).
%\]
%Since a product of two non-constant monomials is again a non-constant monomial,  $\tilde{g}_4({\bf X}_{ i}, {\bf X}_{j}, {\bf X}_{k}, {\bf X}_{l})^2$ equals a sum of finitely many terms each having the form
%\[
%\sum_{\substack{1 \leq p_d< q_d \leq m \\ d= 1, 2}} (\text{a non-constant monomial in }  \left(\rho_{p_dq_d}, {\bf X}_{p_dq_d, i}, {\bf X}_{p_dq_d, j}, {\bf X}_{p_dq_d, k}, {\bf X}_{p_dq_d, l})_{d = 1, 2} \right).
%\]
%In light of \thmref{isserlis},  an expectation of the preceding display must be bounded by
%\[
%\sum_{\substack{1 \leq p_d , q_d \leq m\\d = 1, 2}} |\rho_{p_1q_1}| = \sum_{\substack{1 \leq p_d , q_d \leq m\\d = 1, 2 \\ \cup_{d = 1}^2 |\{p_d, q_d\}|= 4}} |\rho_{pq}| + \sum_{\substack{1 \leq p_d , q_d \leq m\\d = 1, 2 \\ \cup_{d = 1}^2 |\{p_d, q_d\}|\leq 3}} |\rho_{pq}| \lesssim  m^3 (\|\cormat - \iden\|_F + 1)
%\]
%up to constants. The above observations with \eqref{zeta4lesssim} give the bound in \eqref{zeta4est}.

Now we get to \eqref{zeta1est}-\eqref{zeta3est}. The functions $g_c(\cdot)$, $c = 1, \dots, 3$ for the kernel $h$  can be found by simply conditioning and taking expectation using \thmref{isserlis}. With the help of \texttt{mathematica}, they are found to be
\begin{align*}
&g_1 ({\bf X}_i) \\
=
&\sum_{p < q }\frac{1}{4n}\left\{(X_{pi}X_{qi} - \rho_{pq})^2 - (1 + \rho_{pq}^2)\right\} - \\
&\sum_{p < q } \frac{1 }{4 n} \left \{ (1 + \rho_{pq}^2)(X^2_{pi} +X^2_{qi} - 2) + 8 \rho_{pq}(X_{pi} X_{qi} - \rho_{pq})\right\}+\\
  &\sum_{p < q}t_1(X_{pi} ,X_{qi} ,  \rho_{pq})O(n^{-2})\\
:= &  \ \ \tilde{g}_1 ({\bf X}_i) + \sum_{p < q}t_1(X_{pi} ,X_{qi} ,  \rho_{pq})O(n^{-2}),
\end{align*}
\begin{align*}
&g_2 ({\bf X}_i, {\bf X}_j) \\
= & \ \ \frac{1}{4n} \sum_{p < q }\left\{(X_{pi}X_{qi} - \rho_{pq})^2+ (X_{pj}X_{qj} - \rho_{pq})^2 - 2(1 + \rho_{pq}^2)\right\} - \\
& \frac{1 }{12 n} \sum_{p < q }\Bigl\{\bigl[ \sum_{\pi \in \mathcal{S}_2} (X_{p \pi( i)}^2 - 1)(  X_{p \pi(j)}X_{q\pi( j)} - \rho_{pq})^2 +\\
 &2 (X_{p \pi( i)}^2 - 1) (X_{p \pi(i)}X_{q\pi( i)} - \rho_{pq})(X_{p \pi(j)}X_{q\pi( j)} - \rho_{pq})\bigr]+\\
&2\bigl[ (1+ \rho_{pq}^2)(X_{pi}^2 + X_{pj}^2 - 2) + 4 \rho_{pq} (X_{pi}X_{qi} + X_{pj}X_{qj}  - 2 \rho_{pq})\bigr]\Bigr\} - \\
& \frac{1 }{12 n} \sum_{p < q } \Bigl\{\bigl[ \sum_{\pi \in \mathcal{S}_2} (X_{q \pi( i)}^2 - 1)(  X_{p \pi(j)}X_{q\pi( j)} - \rho_{pq})^2 +\\
 &2 (X_{q \pi( i)}^2 - 1) (X_{p \pi(i)}X_{q\pi( i)} - \rho_{pq})(X_{p \pi(j)}X_{q\pi( j)} - \rho_{pq})\bigr]+\\
&2\bigl[ (1+ \rho_{pq}^2)(X_{qi}^2 + X_{qj}^2 - 2) + 4 \rho_{pq} (X_{pi}X_{qi} + X_{pj}X_{qj}  - 2 \rho_{pq})\bigr]\Bigr\} + \\
&\frac{1}{12n}\sum_{p< q} \Bigl\{ 4(X_{pi}X_{qi } - \rho_{pq})(X_{pj}X_{qj } - \rho_{pq}) \\
&+ 2(1 + \rho_{pq}^2)(X^2_{pi} - 1)(X^2_{pj} - 1) \\
&+ 8 \rho_{pq}[ (X_{pi}^2 -1 )(X_{pj}X_{qj} - \rho_{pq}) +(X_{pj}^2 -1 )(X_{pi}X_{qi} - \rho_{pq})] \Bigr\}+ \\
&\frac{1}{12n}\sum_{p< q} \Bigl\{ 4(X_{pi}X_{qi } - \rho_{pq})(X_{pj}X_{qj } - \rho_{pq}) \\
&+ 2(1 + \rho_{pq}^2)(X^2_{qi} - 1)(X^2_{qj} - 1) \\
&+ 8 \rho_{pq}[ (X_{qi}^2 -1 )(X_{pj}X_{qj} - \rho_{pq}) +(X_{qj}^2 -1 )(X_{pi}X_{qi} - \rho_{pq})] \Bigr\}+ \\
 &O(n^{-2}) \sum_{p < q}t_2({ X}_{pi} , X_{qi} ,  X_{pj} ,  X_{qj} , \rho_{pq})\\
:= &\ \ \tilde{g}_2 ({\bf X}_i , {\bf X}_j) + O(n^{-2}) \sum_{p < q}t_2({ X}_{pi} , X_{qi} ,  X_{pj} ,  X_{qj} , \rho_{pq})
\end{align*}
and
\begin{align*}
&g_3({\bf X}_i, {\bf X}_j, {\bf X}_k)\\
= &\sum_{p  < q} \frac{-1}{12} \sum_{\pi \in \mathcal{S}_3}(X^2_{p \pi(i)} - 1)(X_{p\pi(j)}X_{q\pi(j)}- \rho_{pq})(X_{p\pi(k)}X_{q\pi(k)}- \rho_{pq}) + \\
&\sum_{p  < q} \frac{-1}{12} \sum_{\pi \in \mathcal{S}_3}(X^2_{q \pi(i)} - 1)(X_{p\pi(j)}X_{q\pi(j)}- \rho_{pq})(X_{p\pi(k)}X_{q\pi(k)}- \rho_{pq}) + \\
&O(n^{-1}) \sum_{1 \leq p < q \leq m}t_3({ X}_{pi} , X_{qi} ,  X_{pj} ,  X_{qj} , X_{pk} ,  X_{qk} ,\rho_{pq})\\
:= &\tilde{g}_3 ({\bf X}_i, {\bf X}_j, {\bf X}_k) + O(n^{-1}) \sum_{1 \leq p < q \leq m}t_3({ X}_{pi} , X_{qi} ,  X_{pj} ,  X_{qj} , X_{pk} ,  X_{qk} ,\rho_{pq})
\end{align*}
Above,  $t_1 (\cdot), t_2(\cdot)$ and $t_3(\cdot)$ are simply three fixed polynomials in their respective arguments, and their forms are irrelevant to us. $\tilde{g}_1 (\cdot)$, $\tilde{g}_2 (\cdot)$ and $\tilde{g}_3 (\cdot)$ simply collect the terms that do not involve $t_1 (\cdot), t_2(\cdot)$ and $t_3(\cdot)$,  respectively. Note that by the same fact that $2 |ab| \leq a^2 + b^2$ for  $a, b \in \mathbb{R}$,
\begin{align}
\zeta_1 &\lesssim \mathbb{E}\left[\tilde{g}_1 ({\bf X}_i)^2\right] + O\left( \frac{m^4}{n^4}\right) \label{lastlineofzeta1}, \\
\zeta_2 & \lesssim \mathbb{E}\left[\tilde{g}_2 ({\bf X}_i, {\bf X}_j)^2\right] + O\left( \frac{m^4}{n^4}\right),  \label{lastlineofzeta2}\\
\zeta_3  & \lesssim  \mathbb{E}\left[\tilde{g}_3 ({\bf X}_i, {\bf X}_j, {\bf X}_k)^2\right] + O\left( \frac{m^4}{n^2}\right). \label{lastlineofzeta3}
\end{align}
Note that in the definition of $\tilde{g}_1$, there is a leading factor of order $n^{-1}$. By applying \thmref{isserlis} with the help of \texttt{mathematica}, we  find that $ \mathbb{E}[\tilde{g}_1 ({\bf X}_i)^2]$ is a finite sum of terms each, up to a factor of order $n^{-2}$,  can be bounded by one of the forms:
\begin{equation} \label{bddoneofforms}
\sum_{1 \leq r < s \leq m}\sum_{1 \leq p < q \leq m} \rho_{pq}^2 , \quad  \sum_{1 \leq p, q, r, s \leq m} \rho_{pq}^2\rho_{rs}^2, \sum_{1 \leq p, q, r, s \leq m} |\rho_{pq}\rho_{rs}\rho_{pr}\rho_{qs}|.
\end{equation}
We leave it for the reader to check that with the two estimates in \lemref{threebounds} and the familiar trick of decomposing a sum according to the cardinality of an index set as in  \eqref{DefSk},  the forms in \eqref{bddoneofforms} can all  be bounded by
\[
\|\cormat - \iden\|_F^4 + m^2 (1 + \|\cormat - \iden\|_F^2)
\]
up to constants, and hence from \eqref{lastlineofzeta1} we obtain the estimate for $\zeta_1$ in \eqref{zeta1est}.  By the same token, with the help of \texttt{Mathematica} we observe that
\begin{align*}
 \mathbb{E}\left[\tilde{g}_2 ({\bf X}_i, {\bf X}_j) ^2\right] &\lesssim O(n^{-2})  \sum_{ 1 \leq p, q, r, s \leq m} |\rho_{pq}| \\
 \mathbb{E}\left[\left(\tilde{g}_3 ({\bf X}_i, {\bf X}_j, {\bf X}_k) \right)^2\right]  &\lesssim   \sum_{1 \leq p, q, r, s \leq m} \rho_{pq}^2\rho_{rs}^2 +\sum_{1 \leq p, q, r, s \leq m} |\rho_{pq}\rho_{rs}\rho_{pr}\rho_{qs}|,
\end{align*}
again by the index set decomposition trick and \lemref{threebounds} we have
\begin{align}
 \mathbb{E}\left[\tilde{g}_2 ({\bf X}_i, {\bf X}_j) ^2\right] &\lesssim n^{-2}m^3 (\|\cormat - \iden\|_F + 1) \label{tildeg2expbound} \text{ and }\\
 \mathbb{E}\left[\left(\tilde{g}_3 ({\bf X}_i, {\bf X}_j, {\bf X}_k) \right)^2\right]  &\lesssim    \|\cormat - \iden\|_F^4 + m^2 (1 + \|\cormat - \iden\|_F^2) \label{tildeg3expbound}.
\end{align}
Collecting \eqref{lastlineofzeta2},\eqref{lastlineofzeta3}, \eqref{tildeg2expbound} and \eqref{tildeg3expbound} gives us \eqref{zeta2est} and \eqref{zeta3est}.
%and by the computation of \texttt{mathematica} we observe that the first term on the right hand side of the preceding display is a sum of finitely many terms, and up to a factor of $O(n^{-2})$, each can be bounded by
%\[
%\sum_{ 1 \leq p, q, r, s \leq m} |\rho_{pq}| \lesssim m^3 (\|\cormat - \iden\|_F + 1)
%\]
%based  on our sum decomposing trick, giving us \eqref{zeta2est}.
%
%\[
% \quad  \sum_{1 \leq p, q, r, s \leq m} \rho_{pq}^2\rho_{rs}^2, \sum_{1 \leq p, q, r, s \leq m} \rho_{pq}\rho_{rs}\rho_{pr}\rho_{qs}
%\]

\end{proof}

\bibliographystyle{ba}

\bibliography{OptSchott_bib}

\end{document}